\documentclass[10pt,letterpaper]{article}
\makeatletter
\usepackage{amscd}
\usepackage{amsmath}
\usepackage{latexsym}
\usepackage{amsfonts}
\usepackage{amssymb}
\usepackage{amsthm}
\usepackage{graphicx}
\newtheorem{thm}{Theorem}[section]
\newtheorem{lem}{Lemma}[section]
\newtheorem{rem}{Remark}[section]

\newcommand{\be}{\begin{eqnarray}}
\newcommand{\ee}{\end{eqnarray}}
\newcommand{\mr}{\mathbb{R}}

\newcommand{\bes}{\begin{eqnarray*}}
\newcommand{\ees}{\end{eqnarray*}}

 \numberwithin{equation}{section}

\begin{document}
\begin{titlepage}
\title{\bf LOCAL EXISTENCE OF STRONG SOLUTIONS TO THE $k-\varepsilon$ MODEL EQUATIONS FOR TURBULENT FLOWS}
\author{Baoquan Yuan\thanks{Corresponding Author: B. Yuan}\ and Guoquan Qin
       \\ School of Mathematics and Information Science,
       \\ Henan Polytechnic University,  Henan,  454000,  China.\\
        (bqyuan@hpu.edu.cn, qingguoquan@163.com)
          }
\date{}
\end{titlepage}
\maketitle
\begin{abstract}
 In this paper, we are concerned with the local existence of strong solutions
 to the $k-\varepsilon$ model equations for turbulent flows in a bounded
 domain $\Omega$$\subset$ $\mathbb{R}^{3}$. We prove the existence of unique local
 strong solutions under the assumption that  turbulent kinetic energy and the initial density
 both have lower bounds away from zero.

\end{abstract}

\vspace{.2in} {\bf Key words:}\quad $k-\varepsilon$ model equations,
strong solutions, local well-posedness.



\section{Introduction}
\setcounter{equation}{0}

Turbulence is a natural phenomenon which occurs inevitably when the Reynolds number
of flows becomes high enough($10^{6}$ or more).
In this paper, we consider the $k-\varepsilon$ model equations \cite{B-G, L-S} for turbulent flows in a  bounded
 domain $\Omega$$\subset$ $\mathbb{R}^{3}$ with smooth boundary,
\begin{eqnarray}
&&\rho_{t}+\nabla\cdot(\rho u)=0,\label{1.1}\\
&&(\rho u)_{t}+\nabla\cdot(\rho u\otimes u)-\Delta u-\nabla (\nabla\cdot u)+\nabla p=-\frac{2}{3}\nabla(\rho k),\label{1.2}\\
&&(\rho h)_{t}+\nabla\cdot(\rho uh)-\Delta h=p_{t}+u\cdot\nabla p+S_{k},\label{1.3}\\
&&(\rho k)_{t}+\nabla\cdot(\rho uk)-\Delta k=G-\rho \varepsilon,\label{1.4}\\
&&(\rho \varepsilon)_{t}+\nabla\cdot(\rho u\varepsilon)-\Delta \varepsilon=
\frac{C_{1}G\varepsilon}{k}-\frac{C_{2}\rho\varepsilon^{2}}{k},\label{1.5}\\
&&(\rho,u,h,k,\varepsilon)(x,0)=(\rho_{0}(x),u_{0}(x),h_{0}(x),k_{0}(x),\varepsilon_{0}(x)),\label{1.6}\\
&&\bigg(u\cdot \overrightarrow{n},h,
\frac{\partial k}{\partial\overrightarrow{n}},\frac{\partial \varepsilon}{\partial\overrightarrow{n}}\bigg)|_{\partial\Omega}=(0,0, 0, 0),\label{boundary}
\end{eqnarray}

with
\begin{eqnarray}
&&S_{k}=\bigg[\mu\bigg(\frac{\partial u^{i}}{\partial x_{j}}+\frac{\partial u^{j}}
{\partial x_{i}}\bigg)-\frac{2}{3}\delta_{ij}\frac{\partial u^{k}}
{\partial x_{k}}\bigg]\frac{\partial u^{i}}{\partial x_{j}}
+\frac{\mu_{t}}{\rho^{2}}\frac{\partial p}
{\partial x_{j}}\frac{\partial \rho}{\partial x_{j}},\label{1.7}
\end{eqnarray}
\begin{eqnarray}
&&G=\frac{\partial u^{i}}{\partial x_{j}}\bigg[\mu_{e}\bigg(\frac{\partial u^{i}}
{\partial x_{j}}+\frac{\partial u^{j}}{\partial x_{i}}\bigg)-\frac{2}{3}\delta_{ij}\bigg(\rho k+
\mu_{e}\frac{\partial u^{k}}{\partial x_{k}}\bigg)\bigg],\label{1.8}\\
&&p=\rho^{\gamma},\label{1.9}
\end{eqnarray}
where $\delta_{ij}=0$ if $i\neq j$, $\delta_{ij}=1$ if $i=j$, and $\mu,$ $\mu_{t},$ $\mu_{e},$ $C_{1}$ and $C_{2}$
are five positive constants satisfying $\mu+\mu_{t}=\mu_{e},$ and $\overrightarrow{n}$ is the unit outward normal
to $\partial\Omega$.

The equations (\ref{1.1})-(\ref{1.9}) are derived from combining the effect
of turbulence on time-averaged Navier-Stokes equations with the $k-\varepsilon$\
model equations. The unknown functions $\rho,$ $u,$ $h,$ $k$  and $\varepsilon$ denote
the density, velocity, total enthalpy, turbulent kinetic energy and the rate of viscous
dissipation of turbulent flows, respectively. The expression of the pressure $p$ has been
simplified here, which indeed has  no bad effect on our study.

In partial differential equations, $k-\varepsilon$ equations belong
to the compressible ones. In this regard, we will refer to the
classical compressible  Navier-Stokes equations and compressible
MHD equations, which are also research mainstreams, to carry out our
study.

For compressible isentropic Navier-Stokes equations, the first
question provoking our interest
is the existence of the weak solutions. P. L. Lions \cite{Li1, Li2} proved
the global existence of weak solutions under the condition that $\gamma>\frac{3n}{n+2}$,
where $\gamma$ is the same as in (\ref{1.9}) and $n$ is the dimension of space.
Later, Feireisl \cite{F1, F2} improved his result to $\gamma>\frac{n}{2}$.
The condition satisfied by $\gamma$ is to prove the existence of renormalized
solutions, which was introduced by DiPerna and Lions \cite{D-L}.
When the initial data are general small perturbations of non-vacuum resting state,
Hoff \cite{H2} proved the global existence of weak solutions provided $\gamma>1$.
The existence of strong solutions is another problem provoking our interest
in the research of Navier-Stokes equations. It has been proved that the density will be away from
vacuum at least  in a small time  provided the initial density
is positive. If the initial data have better regularity, the compressible isentropic
Navier-Stokes equations will admit  unique  local strong solutions under
various boundary conditions \cite{C-C-K, C-K1, C-K2, S-S}. However, when initial
vacuum is allowed, it was shown recently in \cite{C-C-K} that the isentropic ones will have
local strong solutions in the case  that some  compatibility
conditions are satisfied initially. H. J. Choe and H. S. Kim \cite{C-K3} obtained
the unique local strong solutions for full compressible polytropic Navier-Stokes
equations under the similar condition as in
\cite{C-C-K}. In  \cite{C-K3},  the technic the authors used is mainly the
standard iteration argument and the key point of their success is the estimate
for the $L^{2}$ norm of the gradient of pressure. In the process
of studying the condition of  local solutions becoming global
ones, Z. P. Xin \cite{X} proved that the smooth solutions will blow up in finite
time when initial vacuum is allowed.

As for compressible MHD equations, the research directions, which mainly
contain first the existence of weak and strong solutions and second the condition
of weak solutions becoming strong or even classical ones and the
local becoming  global ones, are similar
to that of Navier-Stokes equations. For example, Hu and Wang \cite{H-W1, H-W2, H-W3} obtained
 the local existence of weak solutions to the compressible isentropic MHD equations.
Rozanova \cite{R} proved the local existence of classical solutions
to the compressible barotropic MHD equations provided both
the mass and energy are finite. J. S. Fan and W. H. Yu in \cite{F-Y} proved
the existence and uniqueness of strong solutions to the full
compressible MHD equations. The method used by J. S. Fan and W. H. Yu \cite{F-Y}
 is similar to that in \cite{C-K3}, for example, they are both dependent
 on the standard iteration argument and the estimate for
 the $L^{2}$ norm of the gradient of pressure.

 Under the the hypothesis of the existence of local-in-time smooth solution, the authors of \cite{B-G}
 prove the existence of small data smooth solution in $\mr^3$.
 In this paper, we consider the local-in-time existence of strong solutions to the
 $k-\varepsilon$ model equations (\ref{1.1})-(\ref{1.9})  in a bounded domain $\Omega\subset \mathbb{R}^{3}$.
 Our method is similar to that in \cite{F-Y} and \cite{C-K3}.
 However, in the process of applying the method to  $k-\varepsilon$ model equations,
 we find that the regularity of the solutions should be higher, which is induced
 by higher nonlinearity than compressible Navier-Stokes equations and compressible
 MHD equations,  than that in  \cite{F-Y} and \cite{C-K3}.
In fact, when we make the difference of the $n-th$
and the $(n+1)-th$ of equation (\ref{2.3}) and integrating the result,  it inevitably comes out the term
$\int\partial_{j}\overline{\rho}^{n+1}\partial_{j}\rho^{n+1}\cdot\overline{h}^{n+1}$.
Therefore, we have to use integration by parts,
which leads to two terms as $\int\overline{\rho}^{n+1}\partial_{j}\partial_{j}{\rho}^{n+1}\cdot\overline{h}^{n+1}$
 and $\int\overline{\rho}^{n+1}\partial_{j}{\rho}^{n+1}\cdot\partial_{j}\overline{h}^{n+1}$.
 Then, by H\"{o}lder and Young's inequalities, it turns out that $\|\nabla^{2}\rho^{n+1}\|_{L^{3}}$
 and $\|\nabla\rho^{n+1}\|_{L^{\infty}}$ should be bounded. Thus, we need $\|\rho\|_{H^{3}}$ be bounded
 for a priori estimates. Therefore, from the mass equation enough regularity of the velocity field should be imposed. Moreover,
due to the strong-coupling property of $k-\varepsilon$ equations, we need
corresponding high regularity of unknown functions $k$ and $\varepsilon$.

In a word, the high nonlinearity of $k-\varepsilon$ equations leads to the
necessity of high regularity of some unknown functions and thus leads to much difficulties
 for the a priori estimates. Besides, physically, when the turbulent kinetic energy $k$ vanishes, the turbulence will disappear and the
 $k-\varepsilon$ model equations will degenerate  into the Navier-Stokes equations, therefore, without loss of generality, we
  assume throughout this paper that the turbulent kinetic energy $k$ has a positive lower bound away from zero
 , namely, $0<m<k$ with $m$ a constant.

 To conclude this introduction, we give the outline of the rest of this paper: In section 2, we consider a linearized problem
of the $k-\varepsilon$ equations and derive some local-in-time estimates
for the solutions of the linearized problem. In section 3, we prove
the existence theorem of the local strong solutions of the original nonlinear
problem.


\section{A priori estimates for a linearized problem }
\setcounter{equation}{0}

Using density equation (\ref{1.1}),
we could change  (\ref{1.1})-(\ref{1.9}) into the following equivalent form :
\begin{eqnarray}\label{equivelence}\begin{cases}
\rho_{t}+\nabla\cdot(\rho u)=0,\\
\rho u_{t}+\rho u\cdot\nabla u-\Delta u-\nabla \mbox{div} u+\nabla p=-\frac{2}{3}\nabla(\rho k),\\
\rho h_{t}+\rho u\cdot\nabla h-\Delta h=p_{t}+u\cdot\nabla p+S_{k},\\
\rho k_{t}+\rho u\cdot\nabla k-\Delta k=G-\rho \varepsilon,\\
\rho \varepsilon_{t}+\rho u\cdot\nabla \varepsilon-\Delta \varepsilon=
\frac{C_{1}G\varepsilon}{k}-\frac{C_{2}\rho\varepsilon^{2}}{k},\\
(\rho,u,h,k,\varepsilon)(x,0)=(\rho_{0}(x),u_{0}(x),h_{0}(x),k_{0}(x),\varepsilon_{0}(x)),\\
(u\cdot \overrightarrow{n},h,
\frac{\partial k}{\partial\overrightarrow{n}},\frac{\partial \varepsilon}{\partial\overrightarrow{n}})|_{\partial\Omega}=(0, 0, 0, 0).
\end{cases}
\end{eqnarray}

Then, we consider the following linearized problem of (\ref{equivelence}):
\begin{eqnarray}
&&\rho_{t}+\nabla\cdot(\rho v)=0,\label{2.1}\\
&&\rho u_{t}+\rho v\cdot\nabla u-\Delta u-\nabla \mbox{div} u+\nabla p=-\frac{2}{3}\nabla(\rho \pi),\label{2.2}\\
&&\rho h_{t}+\rho v\cdot\nabla h-\Delta h=p_{t}+u\cdot\nabla p+S_{k}^{'},\label{2.3}\\
&&\rho k_{t}+\rho v\cdot\nabla k-\Delta k=G^{'}-\rho \theta,\label{2.4}\\
&&\rho \varepsilon_{t}+\rho v\cdot\nabla \varepsilon-\Delta \varepsilon=
\frac{C_{1}G^{'}\theta}{\pi}-\frac{C_{2}\rho\theta^{2}}{\pi},\label{2.5}\\
&&(\rho,v,h,\pi,\theta)(x,0)=(\rho_{0}(x),u_{0}(x),h_{0}(x),k_{0}(x),\varepsilon_{0}(x)),\\
&&\bigg(v\cdot\overrightarrow{n},h,
\frac{\partial \pi}{\partial\overrightarrow{n}},\frac{\partial \theta}{\partial\overrightarrow{n}}\bigg)|_{\partial\Omega}=(0, 0, 0, 0).\label{2.55}
\end{eqnarray}

with
\begin{eqnarray*}
&&S_{k}^{'}=\bigg[\mu\bigg(\frac{\partial v^{i}}{\partial x_{j}}+
\frac{\partial v^{j}}{\partial x_{i}}\bigg)-\frac{2}{3}\delta_{ij}
\frac{\partial v^{k}}{\partial x_{k}}\bigg]\frac{\partial v^{i}}{\partial x_{j}}\
+\frac{\mu_{t}}{\rho^{2}}\frac{\partial p}{\partial x_{j}}\frac{\partial \rho}{\partial x_{j}}, \\
&&G^{'}=\frac{\partial v^{i}}{\partial x_{j}}\bigg[\mu_{e}
\bigg(\frac{\partial v^{i}}{\partial x_{j}}+\frac{\partial v^{j}}
{\partial x_{i}}\bigg)-\frac{2}{3}\delta_{ij}\bigg(\rho \pi+
\mu_{e}\frac{\partial v^{k}}{\partial x_{k}}\bigg)\bigg],
\end{eqnarray*}
where $v,$ $\pi$ and $\theta$ are known quantities on $(0,T_{1})\times\Omega$ with $T_{1}>0$.

Here we also impose the following regularity conditions on the initial data:

\be\label{1} \begin{cases}
 0<m<\rho_{0},\ \rho_{0}\in H^{3}(\Omega),\\
 u_{0}\in  H^{3}(\Omega),\\
 (h_{0}, k_{0}, \varepsilon_{0})\in   H^{2}(\Omega),\\
  \bigg(u_{0}\cdot\overrightarrow{n},h_{0},
\frac{\partial k_{0}}{\partial\overrightarrow{n}},\frac{\partial \varepsilon_{0}}{\partial\overrightarrow{n}}\bigg)\bigg|_{\partial\Omega}=(0, 0, 0, 0),\\
0<m<k_{0}.\\
\end{cases} \ee


For the known quantities $v, \pi, \theta$, we assume that $v(0)=u_{0}, \pi(0)=k_{0}, \theta(0)=\varepsilon_{0}$ and
\be\begin{cases}\label{4}
 \mathop{\sup}_{0\leq t\leq T_{2}}(\|v\|_{H^{1}}
 +\|\pi\|_{H^{1}}+\|\theta\|_{H^{1}})\\
 +\int_{0}^{T_{2}}(\|\pi\|_{H^{3}}^{2}
 +\|v_{t}\|_{H^{1}}^{2}+\|\pi_{t}\|_{H^{1}}^{2}+\|\theta_{t}\|_{H^{1}}^{2})\mbox{d}t\leq c_{1},\\
 \mathop{\sup}_{0\leq t\leq T_{2}}\|v\|_{H^{2}}\leq c_{2},\\
 \mathop{\sup}_{0\leq t\leq T_{2}}\|v\|_{H^{3}}\leq c_{3},\\
  \int_{0}^{T_{2}}\|v\|_{H^{4}}^{2}\mbox{d}t\leq c_{4},\\
   \mathop{\sup}_{0\leq t\leq T_{2}}\|\pi\|_{H^{2}}\leq c_{5},\\
   \mathop{\sup}_{0\leq t\leq T_{2}}\|\theta\|_{H^{2}}\leq c_{6}
   \end{cases}
\ee
for some fixed constants $c_i$ satisfying $1<c_0<c_i(i=1,2,\cdots,6)$ and some time $T_{2}>0.$
Where
\be
\nonumber &&c_{0}=2+\|(\rho_{0}, u_{0})\|_{H^{3}}+\|(h_{0},k_{0},\varepsilon_{0})\|_{H^{2}}.
\ee
And for simplicity, we set another small time $T$ as $T$=min\{$c_{0}^{-6\gamma-16}c_{1}^{-10}c_{2}^{-8}c_{3}^{-8}c_{4}^{-2}
c_{5}^{-2}c_{6}^{-4},$  $T_{1},$ $T_{2}$\} and all of the $T$ in section 2 are defined as this.
\begin{rem}
Here it should be emphasized that throughout this paper, C denotes a
generic positive constant which is only dependent on $m,  \gamma$
and $|\Omega|$, but independent of $c_i\ (i=0,1,2,\cdots,6)$.
\end{rem}

\begin{rem}
From the physical viewpoint, we assume that the turbulent kinetic energy $k$ has a positive lower bound away from zero
 , namely, $0<m<k$ with $m$ a constant. We do not know whether $0<m<k$ holds afterwards if its initial value $k_0>m$.
 \end{rem}


Next, we would like to prove the following local existence theorem of the linearized system (\ref{2.1})-(\ref{2.5}).

\begin{thm}\label{lem6}
There exists a unique strong solution $(\rho,u,h,k,\varepsilon)$ to the linearized
problem (\ref{2.1})-(\ref{2.55}) and (\ref{1}) in $[0,T]$ satisfying the estimates
(\ref{conclusion1}) and (\ref{conclusion2}) as well as the regularity
\begin{eqnarray}
&&\rho \in C(0,T;H^{3}),
\rho_{t} \in C(0,T;H^{1}),
u \in C(0,T; H^{3})\cap L^{2}(0,T; H^{4}),\nonumber\\
&&u_{t}\in L^{2}(0,T;H^{1}),
k\in C(0,T; H^{2})\cap L^{2}(0,T;H^{3}),
k_{t}\in L^{2}(0,T;H^{1}),\nonumber\\
&&\varepsilon\in C(0,T; H^{2}),
\varepsilon_{t}\in L^{2}(0,T;H^{1}),
h\in C(0,T; H^{2}),
h_{t}\in L^{2}(0,T;H^{1}),\nonumber\\
&&(\sqrt{\rho}u_{t},\sqrt{\rho}k_{t},\sqrt{\rho}\varepsilon_{t},\sqrt{\rho}h_{t})\in L^{\infty}(0,T;L^{2}).\nonumber
\end{eqnarray}
\end{thm}

In the following part, we decompose the proof of Theorem \ref{lem6}  into some lemmas.

\begin{lem}\label{lem1}
There exists a unique strong solution $\rho$ to the  linear transport
problem (\ref{2.1})  and (\ref{1}) such that
\begin{eqnarray}
\rho\geq \frac{m}{e},\
\|\rho\|_{H^{3}(\Omega)}\leq Cc_{0},\  \|\rho_{t}\|_{H^{1}(\Omega)}\leq Cc_{0}c_{2}\label{routH3}
\end{eqnarray}
for $0\leq t\leq T$.
\end{lem}
\begin{proof}
First, applying the particle trajectory method to equation $(\ref{2.2})$, we easily deduce
\be
\rho\geq\rho_{0}\exp\bigg(-\int_{0}^{T}\|\nabla v\|_{L^{\infty}}\mbox{d}t\bigg)
\geq\rho_{0}\exp(-c_{3}T)\geq\frac{\rho_{0}}{e}\geq\frac{m}{e}\nonumber
\ee
and thus
\be
\frac{1}{\rho}\leq \frac{e}{m}\leq C.\nonumber
\ee
Second, by simple calculation, we have
\be
\frac{\mbox{d}}{\mbox{d}t}\|\rho\|_{H^{3}}\leq C\|v\|_{H^{3}}\|\rho\|_{H^{3}}+C\|\nabla^{4}v\|_{L^{2}},\nonumber
\ee
applying Gronwall and H\"{o}lder's inequalities, one gets
\be
\|\rho\|_{H^{3}}\leq \bigg[\exp\bigg(C\int_{0}^{t}\|v\|_{H^{3}}\mbox{d}t\bigg)\bigg]\bigg(\|\rho_{0}\|_{H^{3}}
+C\int_{0}^{t}\|v\|_{H^{4}}\mbox{d}t\bigg)\leq Cc_{0}\nonumber
\ee
for $0\leq t\leq T$.

Next, from equation (\ref{2.1}), one obtains
\be
\|\rho_{t}\|_{H^{1}}=\|\nabla\cdot(\rho v)\|_{H^{1}}\leq C\|\rho\|_{H^{3}}\|v\|_{H^{2}}\leq Cc_{0}c_{2}\nonumber
\ee
for $0\leq t\leq T$.

Thus, we complete the proof of Lemma \ref{lem1}.
\end{proof}

Next, we estimate the velocity field $u$.
\begin{lem}\label{lem2}
There exists a unique strong solution u to the initial boundary value problem (\ref{2.2}) and (\ref{1})  such that
\begin{eqnarray}
&&\|\sqrt{\rho}u_{t}\|_{L^{2}}^{2}
+\| u\|_{H^{1}}^{2}+\int_{0}^{t}\|\nabla u_{t}\|_{L^{2}}^{2}\mbox{d}s \leq Cc_{0}^{5+2\gamma},\ \|u\|_{H^{2}}\leq Cc_{0}^{\frac{5}{2}+3\gamma}c_{1}^{2},\label{u1}\\
&&\|u\|_{H^{3}}\leq Cc_{0}^{\frac{13}{2}+3\gamma}c_{1}^{4}c_{2}c_{5}, \ \int_{0}^{t}\|u\|_{H^{4}}^{2}\mbox{d}s \leq  Cc_{0}^{9+6\gamma}c_{1}^{5}c_{2}^{2}\label{u2}
\end{eqnarray}
for $0\leq t\leq T$.
\end{lem}
\begin{proof}We only need to prove
the estimates. Differentiating equation (\ref{2.2}) with respect to t, then multiplying
both sides of the result by $u_{t}$ and integrating over $\Omega$, we derive that
\begin{eqnarray}
&&\frac{1}{2}\frac{\mbox{d}}{\mbox{d}t}\int {\rho u_{t}^{2}\mbox{d}x}+\|\nabla u_{t}\|_{L^{2}}^{2}+\|\mbox{div} u_{t}\|_{L^{2}}^{2}\nonumber\\
&&=-\int \rho_{t}v\cdot\nabla u\cdot u_{t}-\int \rho v_{t}\cdot\nabla u\cdot u_{t}-2\int\rho v\cdot\nabla u_{t}\cdot u_{t}-\int\nabla p_{t}\cdot u_{t}-\frac{2}{3}\int[\nabla(\rho\pi)]_{t}\cdot u_{t}\nonumber\\
&&=I_{1}+I_{2}+I_{3}+I_{4}+I_{5},\label{u eqn}
\end{eqnarray}
where we have used  equation (\ref{2.1}) and integration by parts.
We will estimate $I_{i}\ (i=1,2,\cdots,5)$ item by item.

First, because $\rho$ has lower bound away from zero, we easily deduce $\|u_{t}\|_{L^{2}}\leq C\|\sqrt{\rho}u_{t}\|_{L^{2}}$.
  Therefore, using H\"{o}lder, Sobolev and Young's inequalities and (\ref{4}), we have
\begin{eqnarray}
&&I_{1}\leq C\|v\|_{L^{\infty}}\|\rho_{t}\|_{L^{3}}\|\nabla u\|_{L^{2}}\|u_{t}\|_{L^{6}}
\leq C\|v\|_{L^{\infty}}\|\rho_{t}\|_{L^{3}}\|\nabla u\|_{L^{2}}(\|\sqrt{\rho}u_{t}\|_{L^{2}}+\|\nabla u_{t}\|_{L^{2}})\nonumber\\
&&\leq Cc_{0}^{2}c_{2}^{4}\|\nabla u\|_{L^{2}}^{2}+C\|\sqrt{\rho}u_{t}\|_{L^{2}}^{2}+\frac{1}{8}\|\nabla u_{t}\|_{L^{2}}^{2},\label{I1}
\end{eqnarray}
\begin{eqnarray}
 I_{3}\leq C\|\rho\|_{L^{\infty}}^{\frac{1}{2}}\|v\|_{L^{\infty}}\|\nabla u_{t}\|_{L^{2}}\|\sqrt{\rho}u_{t}\|_{L^{2}}
 \leq Cc_{0}c_{2}^{2}\|\sqrt{\rho}u_{t}\|_{L^{2}}^{2}+\frac{1}{8}\|\nabla u_{t}\|_{L^{2}}^{2},\label{I3}
 \end{eqnarray}

\begin{eqnarray}
I_{2}
\leq C\|\rho\|_{L^{\infty}}^{\frac{1}{2}}\|v_{t}\|_{L^{6}}\|\nabla u\|_{L^{3}}\|\sqrt{\rho}u_{t}\|_{L^{2}}
\leq C\eta^{-1}c_{0}\|\nabla u\|_{L^{3}}^{2}+\eta\|v_{t}\|_{H^{1}}^{2}\|\sqrt{\rho}u_{t}\|_{L^{2}}^{2},\label{I2}
\end{eqnarray}
where $\eta>0$ is a small number to be determined later.

Next, to evaluate $\|\nabla u\|_{L^{3}}^{2}$ in (\ref{I2}), we can first use Sobolev's interpolation inequality to get
\be
\|\nabla u\|_{L^{3}}^{2}\leq C\|\nabla u\|_{L^{2}}\|\nabla u\|_{L^{6}}\leq C\|\nabla u\|_{L^{2}}\|\nabla u\|_{H^{1}}.\label{grad u L3}
\ee
Then, applying the standard elliptic regularity result to equation (\ref{2.2}) and using (\ref{grad u L3}), we have
 \begin{eqnarray}
 \|\nabla u\|_{H^{1}}
 \leq Cc_{0}^{\gamma}(\|\sqrt{\rho}u_{t}\|_{L^{2}}+\|v\|_{L^{6}}\|\nabla u\|_{L^{2}}^{\frac{1}{2}}\|\nabla u\|_{H^{1}}^{\frac{1}{2}}
 +\|\nabla \rho\|_{L^{2}}+\|\nabla\rho\|_{L^{4}}\|\pi\|_{L^{4}}+\|\nabla\pi\|_{L^{2}}),\nonumber
 \end{eqnarray}
 thus Young's inequality and (\ref{4}) yield
 \begin{equation}
 \|\nabla u\|_{H^{1}}\leq Cc_{0}^{2\gamma}(\|\sqrt{\rho}u_{t}\|_{L^{2}}+c_{1}^{2}\|\nabla u\|_{L^{2}}+c_{0}c_{1}).\label{u H1}
 \end{equation}
Combining  (\ref{I2}), (\ref{grad u L3}) and (\ref{u H1}) and using Young's inequality, we get
 \begin{eqnarray}
 &I_{2}&\leq C\eta^{-1}c_{0}^{2\gamma+1}(\|\sqrt{\rho}u_{t}\|_{L^{2}}^{2}+c_{1}^{2}\|\nabla u\|_{L^{2}}^{2}
 +c_{0}^{2}c_{1}^{2})
 +\eta\|v_{t}\|_{H^{1}}^{2}\|\sqrt{\rho}u_{t}\|_{L^{2}}^{2}.\label{I2-2}
 \end{eqnarray}

 By integration by parts, we have
 \begin{eqnarray}
 I_{4}=\int p_{t}\mbox{div} u_{t}
 \leq Cc_{0}^{\gamma-1}\|\rho_{t}\|_{L^{2}}\|\nabla u_{t}\|_{L^{2}}\leq Cc_{0}^{2\gamma}c_{2}^{2}+\frac{1}{8}\|\nabla u_{t}\|_{L^{2}}^{2},\label{I4}
 \end{eqnarray}
 \begin{eqnarray}
 &I_{5}&=\frac{2}{3}\int\rho_{t}\pi\nabla\cdot u_{t}-\frac{2}{3}\int\pi_{t}\nabla\rho\cdot u_{t}
 -\frac{2}{3}\int\rho\nabla\pi_{t}\cdot u_{t}\nonumber\\
 &\leq& C\|\rho_{t}\|_{L^{3}}\|\pi\|_{L^{6}}\|\nabla u_{t}\|_{L^{2}}
 +Cc_{0}^{\frac{1}{2}}\|\nabla\rho\|_{L^{3}}\|\pi_{t}\|_{L^{6}}\|\sqrt{\rho}u_{t}\|_{L^{2}}
 +Cc_{0}^{\frac{1}{2}}\|\nabla\pi_{t}\|_{L^{2}}\|\sqrt{\rho}u_{t}\|_{L^{2}}\label{I5}\\
 &\leq& Cc_{0}^{2}c_{1}^{2}c_{2}^{2}+C\eta^{-1}c_{0}^{3}+C\eta\|\pi_{t}\|_{H^{1}}^{2}\|\sqrt{\rho}u_{t}\|_{L^{2}}^{2}
 +\frac{1}{8}\|\nabla u_{t}\|_{L^{2}}^{2}.\nonumber
  \end{eqnarray}
  On the other hand, we easily have
  \be
\frac{\mbox{d}}{\mbox{d}t}\int|\nabla u|^{2}
=2\int \nabla u\cdot\nabla u_{t}\leq\frac{1}{8}\|\nabla u_{t}\|_{L^{2}}^{2}+C\|\nabla u\|_{L^{2}}^{2},\label{grad u L2}
\ee
and
\be
\frac{\mbox{d}}{\mbox{d}t}\int|u|^{2}\leq Cc_{0}^{\frac{1}{2}}\|\sqrt{\rho}u_{t}\|_{L^{2}}\|u\|_{L^{2}}
\leq Cc_{0}\|\sqrt{\rho}u_{t}\|_{L^{2}}^{2}+C\|u\|_{L^{2}}^{2}.\label{u L2}
\ee
  Combining (\ref{u eqn})-(\ref{I3}) and  (\ref{I2-2})-(\ref{u L2}), we get
\begin{eqnarray}
&&\frac{\mbox{d}}{\mbox{d}t}(\|\sqrt{\rho}u_{t}\|_{L^{2}}^{2}+\|u\|_{H^{1}}^{2})+\|\nabla u_{t}\|_{L^{2}}^{2}\\
&&\leq C(c_{0}^{2}c_{2}^{4}+\eta^{-1}c_{0}^{2\gamma+1}c_{1}^{2}+\eta\|\pi_{t}\|_{H^{1}}^{2}
+\eta\|v_{t}\|_{H^{1}}^{2})(\|\sqrt{\rho}u_{t}\|_{L^{2}}^{2}+\|u\|_{H^{1}}^{2})\nonumber\\
&&+C(c_{0}^{2\gamma}c_{1}^{2}c_{2}^{2}+\eta^{-1}c_{0}^{2\gamma+3}c_{1}^{2}),\nonumber
\end{eqnarray}
setting $\eta=\frac{1}{c_{1}}$ and using Gronwall's inequality, we derive
\be
\|\sqrt{\rho}u_{t}\|_{L^{2}}^{2}+\|u\|_{H^{1}}^{2}+\int_{0}^{t}\|\nabla u_{t}\|_{L^{2}}^{2}\mbox{d}s\leq Cc_{0}^{5+2\gamma}\label{u's 1}
\ee
for $0\leq t\leq T$, where we have used the fact that $\mathop{\lim}_{t\rightarrow 0}(\|\sqrt{\rho}u_{t}\|_{L^{2}}^{2}+\|u\|_{H^{1}}^{2})\leq Cc_{0}^{5+2\gamma}$.

Next,   by  (\ref{u H1}) and (\ref{u's 1}), we deduce
\begin{eqnarray}
\|\nabla u\|_{H^{1}}\leq Cc_{0}^{\frac{5}{2}+3\gamma}c_{1}^{2},\label{u H1-2}
\end{eqnarray}
which   implies (\ref{u1}) by (\ref{u's 1}).

Next, we will estimate $\int_{0}^{t}\|u\|_{H^{4}}^{2}\mbox{d}t$.
By the standard elliptic regularity result of equation (\ref{2.2}), we have
\begin{eqnarray}
\|\nabla^{4}u\|_{L^{2}}\leq \|\rho u_{t}\|_{H^{2}}+\|\rho v\cdot\nabla u\|_{H^{2}}+\|\nabla p\|_{H^{2}}+\|\frac{2}{3}\nabla(\rho\pi)\|_{H^{2}}.\label{grad4u}
\end{eqnarray}
By simple calculation, the first term of the right hand side of $(\ref{grad4u})$ can be controlled as
\be
\|\rho u_{t}\|_{H^{2}}\leq C(\|\rho u_{t}\|_{L^{2}}+\|\rho\|_{H^{2}}\|u_{t}\|_{H^{2}})\leq Cc_{0}\|u_{t}\|_{H^{2}}.\label{rouutH2-2}
\ee
In order to estimate $\|\nabla^{2}u_{t}\|_{L^{2}}$,
differentiating equation (\ref{2.2}) with respect to $t$ yields
\be
&&\Delta u_{t}+\nabla \mbox{div} u_{t}\label{grad2ut}\\
&&=\rho_{t}u_{t}+\rho u_{tt}+\rho_{t}v\cdot\nabla u+\rho v_{t}\cdot\nabla u
+\rho v\cdot\nabla u_{t}+\nabla p_{t}+\frac{2}{3}(\nabla\rho_{t}\pi+\rho_{t}\nabla\pi+\nabla\rho\pi_{t}+\rho\nabla\pi_{t}),\nonumber
\ee
applying  the standard elliptic regularity result to (\ref{grad2ut}) and using (\ref{u's 1}), one obtains
\begin{eqnarray}
&&\|\nabla^{2}u_{t}\|_{L^{2}}
\leq C(\|\rho_{t}\|_{L^{4}}\|u_{t}\|_{L^{4}}+\|\rho u_{tt}\|_{L^{2}}+\|\rho_{t}\|_{L^{4}}\|v\|_{L^{\infty}}\|\nabla u\|_{L^{4}}
+\|\rho\|_{L^{\infty}}\|v_{t}\|_{L^{4}}\|\nabla u\|_{L^{4}}\nonumber\\
&&+\|v\|_{L^{\infty}}\|u_{t}\|_{H^{1}}
+\|\rho\|_{H^{2}}^{\gamma}\|\rho_{t}\|_{H^{1}}
+\|\pi\|_{L^{\infty}}\|\rho_{t}\|_{H^{1}}+\|\rho_{t}\|_{L^{4}}\|\nabla\pi\|_{L^{4}}\nonumber\\
&&+\|\nabla\rho\|_{L^{4}}\|\pi_{t}\|_{L^{4}}
+\|\rho\|_{L^{\infty}}\|\nabla\pi_{t}\|_{L^{2}})\nonumber\\
&&\leq C(\|\rho u_{tt}\|_{L^{2}}+c_{0}^{\frac{7}{2}+3\gamma}c_{1}^{2}c_{2}^{2}c_{5}+c_{0}^{\frac{7}{2}+3\gamma}c_{1}^{2}\|v_{t}\|_{H^{1}}
+c_{0}c_{2}\|u_{t}\|_{H^{1}}+c_{0}\|\pi_{t}\|_{H^{1}}),\label{grad2ut2}
\end{eqnarray}
therefore, the key point is to estimate $\|\rho u_{tt}\|_{L^{2}}$. Because we have the fact $\|\rho u_{tt}\|_{L^{2}}\leq C\|\sqrt{\rho} u_{tt}\|_{L^{2}}$, we could first estimate $\|\sqrt{\rho} u_{tt}\|_{L^{2}}$ as follows.

Multiplying both sides of (\ref{grad2ut}) by $u_{tt}$ and integrating the result over $\Omega$ yield
\begin{eqnarray}
&&\int\rho u_{tt}^{2}\mbox{d}x+\frac{1}{2}\frac{\mbox{d}}{\mbox{d}t}\|\nabla u_{t}\|_{L^{2}}^{2}
+\frac{1}{2}\frac{\mbox{d}}{\mbox{d}t}\|\mbox{div} u_{t}\|_{L^{2}}^{2}\nonumber\\
&&=-\int\rho_{t}u_{t}\cdot u_{tt}-\int\rho_{t}v\cdot\nabla u\cdot u_{tt}-\int\rho v_{t}\cdot\nabla u\cdot u_{tt}
-\int\rho v\cdot\nabla u_{t}\cdot u_{tt}-\int\nabla p_{t}\cdot u_{tt}\nonumber\\
&&-\frac{2}{3}\int(\pi\nabla\rho_{t}+\rho_{t}\nabla\pi+\pi_{t}\nabla\rho+\rho\nabla\pi_{t})\cdot u_{tt}
=J_{1}+J_{2}+J_{3}+J_{4}+J_{5}+J_{6}.\label{grad2utt}
\end{eqnarray}
 Using H\"{o}lder, Sobolev and Young's inequalities and (\ref{4}) and (\ref{u's 1}), we get
 \begin{eqnarray}
 &&J_{1}\leq Cc_{0}^{\frac{1}{2}}\|\rho_{t}\|_{L^{3}}\|u_{t}\|_{L^{6}}\|\sqrt{\rho}u_{tt}\|_{L^{2}}
 \leq Cc_{0}^{\frac{1}{2}}\|\rho_{t}\|_{L^{3}}(\|\sqrt{\rho}u_{t}\|_{L^{2}}+\|\nabla u_{t}\|_{L^{2}})\|\sqrt{\rho}u_{tt}\|_{L^{2}}\nonumber\\
 &&\leq Cc_{0}^{3}c_{2}^{2}\|\nabla u_{t}\|_{L^{2}}^{2}+Cc_{0}^{8+2\gamma}c_{2}^{2}+\frac{1}{18}\|\sqrt{\rho}u_{tt}\|_{L^{2}}^{2},\label{J1}
 \end{eqnarray}
  \begin{eqnarray}
  J_{2}\leq Cc_{0}^{\frac{1}{2}}\|\sqrt{\rho}u_{tt}\|_{L^{2}}\|\rho_{t}\|_{L^{3}}
  \|v\|_{L^{\infty}}\|\nabla u\|_{L^{6}}
  \leq Cc_{0}^{8+6\gamma}c_{1}^{4}c_{2}^{4}+\frac{1}{18}\|\sqrt{\rho}u_{tt}\|_{L^{2}}^{2},\label{J2}
  \end{eqnarray}
  \begin{eqnarray}
  J_{3}\leq Cc_{0}^{\frac{1}{2}}\|\sqrt{\rho}u_{tt}\|_{L^{2}}\|v_{t}\|_{L^{3}}\|\nabla u\|_{L^{6}}
  \leq Cc_{0}^{6+6\gamma}c_{1}^{4}\|v_{t}\|_{H^{1}}^{2}+\frac{1}{18}\|\sqrt{\rho}u_{tt}\|_{L^{2}}^{2},\label{J3}
  \end{eqnarray}
  \begin{eqnarray}
  J_{4}\leq Cc_{0}^{\frac{1}{2}}\|v\|_{L^{\infty}}\|\sqrt{\rho}u_{tt}\|_{L^{2}}\|\nabla u_{t}\|_{L^{2}}
  \leq Cc_{0}c_{2}^{2}\|\nabla u_{t}\|_{L^{2}}^{2}+\frac{1}{18}\|\sqrt{\rho}u_{tt}\|_{L^{2}}^{2},\label{J4}
  \end{eqnarray}
  \begin{eqnarray}
  J_{5}\leq Cc_{0}^{\frac{1}{2}}\|\sqrt{\rho}u_{tt}\|_{L^{2}}\|\nabla p_{t}\|_{L^{2}}
  \leq Cc_{0}^{2\gamma+1}c_{2}^{2}+\frac{1}{18}\|\sqrt{\rho}u_{tt}\|_{L^{2}}^{2},\label{J5}
    \end{eqnarray}
    \begin{eqnarray}
    &&J_{6}
    \leq Cc_{0}^{\frac{1}{2}}\|\pi\|_{L^{\infty}}\|\sqrt{\rho}u_{tt}\|_{L^{2}}\|\nabla \rho_{t}\|_{L^{2}}+
    Cc_{0}^{\frac{1}{2}}\|\sqrt{\rho}u_{tt}\|_{L^{2}}\|\nabla\pi\|_{L^{4}}\|\rho_{t}\|_{L^{4}}\nonumber\\
    &&+Cc_{0}^{\frac{1}{2}}\|\sqrt{\rho}u_{tt}\|_{L^{2}}\|\nabla\rho\|_{L^{\infty}}\|\pi_{t}\|_{L^{2}}
    +Cc_{0}^{\frac{1}{2}}\|\sqrt{\rho}u_{tt}\|_{L^{2}}\|\nabla\pi_{t}\|_{L^{2}}\nonumber\\
    &&\leq Cc_{0}^{3}c_{2}^{2}c_{5}^{2}+Cc_{0}^{3}\|\pi_{t}\|_{H^{1}}^{2}+\frac{2}{9}\|\sqrt{\rho}u_{tt}\|_{L^{2}}^{2},\label{J6}
     \end{eqnarray}
 inserting (\ref{J1})-(\ref{J6}) to (\ref{grad2utt}), then integrating the result  over $(0,t)$, we derive
 \be
 \int_{0}^{t}\int_{\Omega}\rho u_{tt}^{2}dx\mbox{d}t+\|\nabla u_{t}\|_{L^{2}}^{2}\leq Cc_{0}^{6+6\gamma}c_{1}^{5}c_{2}^{2},\label{rouutt2}
 \ee
 where we have used equation (\ref{2.2}) to get $\lim_{t\rightarrow0}\|\nabla u_{t}(t)\|_{L^{2}}^{2}\leq Cc_{0}^{2\gamma+4}$.

  So, combining (\ref{rouutH2-2}), (\ref{grad2ut2}) and (\ref{rouutt2}), we obtain
  \be
  \int_{0}^{t}\|\rho u_{t}\|_{H^{2}}^{2}\leq Cc_{0}^{9+6\gamma}c_{1}^{5}c_{2}^{2}.\label{rouutH2}
  \ee
In the following, we shall estimate the rest terms of the inequality
(\ref{grad4u}).

   For the second term of the inequality (\ref{grad4u}), direct calculation yields
\begin{eqnarray}
&&\|\rho v\cdot\nabla u\|_{H^{2}}\leq C\|\rho\|_{H^{2}}\|v\|_{H^{2}}\|u\|_{H^{3}}\leq Cc_{0}c_{2}\|u\|_{H^{3}},\label{rouvgradu}
\end{eqnarray}
therefore, we have to evaluate $\|u\|_{H^{3}}$.
In fact, Applying the standard elliptic regularity result to equation $(\ref{2.2})$, we obtain
\begin{eqnarray}
\|\nabla^{3}u\|_{L^{2}}\leq C(\|\rho u_{t}\|_{H^{1}}+\|\rho v\cdot\nabla u\|_{H^{1}}+\|\nabla p\|_{H^{1}}+\|\nabla(\rho\pi)\|_{H^{1}}),\label{grad3u}
\end{eqnarray}
we could estimate the right hand side of (\ref{grad3u}) item by item.

First,  from (\ref{u's 1}), we have $\|u_{t}\|_{L^{2}}\leq Cc_{0}^{\frac{5}{2}+\gamma}$, thus
\begin{eqnarray}
\|\rho u_{t}\|_{H^{1}}
\leq Cc_{0}\|u_{t}\|_{L^{2}}+\|\nabla\rho\|_{L^{\infty}}\|u_{t}\|_{L^{2}}+Cc_{0}\|\nabla u_{t}\|_{L^{2}}
\leq Cc_{0}^{\frac{7}{2}+\gamma}+Cc_{0}\|\nabla u_{t}\|_{L^{2}}.\label{grad3u1}
\end{eqnarray}
Second, using Sobolev's interpolation inequality and Young's inequality, we get
\begin{eqnarray}
&&\|\rho v\cdot\nabla u\|_{H^{1}}\leq C(\|\rho v\cdot\nabla u\|_{L^{2}}+\|\nabla(\rho v\cdot\nabla u)\|_{L^{2}})\nonumber\\
&&\leq C(c_{0}\|v\|_{L^{\infty}}\|\nabla u\|_{L^{2}}+\|\nabla\rho\|_{L^{\infty}}\|v\|_{L^{\infty}}\|\nabla u\|_{L^{2}}
+c_{0}\|\nabla v\|_{L^{2}}\|\nabla u\|_{L^{2}}^{\frac{1}{4}}\|\nabla^{3} u\|_{L^{2}}^{\frac{3}{4}}\nonumber\\
&&+c_{0}\|v\|_{L^{\infty}}\|\nabla^{2}u\|_{L^{2}})
\leq Cc_{0}^{\frac{13}{2}+3\gamma}c_{1}^{4}c_{2}+\frac{3}{4}\|u\|_{H^{3}}.\label{grad3u2}
\end{eqnarray}
Third, due to (\ref{routH3}), we easily derive
\begin{eqnarray}
\|\nabla p\|_{H^{1}}\leq Cc_{0}^{2}.\label{grad3u3}
\end{eqnarray}
Last, by simple calculation, one gets
\begin{eqnarray}
\|\nabla(\rho\pi)\|_{H^{1}}
\leq C\|\rho\|_{H^{3}}\|\pi\|_{H^{2}}
\leq Cc_{0}c_{5}.\label{grad3u4}
\end{eqnarray}
Combining (\ref{rouutt2}) and (\ref{grad3u})-(\ref{grad3u4}), we deduce
 \begin{eqnarray}
   \|u\|_{H^{3}}
   \leq Cc_{0}^{\frac{13}{2}+3\gamma}c_{1}^{4}c_{2}c_{5}.\label{uH3}
    \end{eqnarray}
  Next, by simple calculation, the third and fourth terms on the right hand side of (\ref{grad4u}) can be estimated as
   \begin{eqnarray}
   \|\nabla p\|_{H^{2}}\leq Cc_{0}^{3},\
   \|\nabla(\rho\pi)\|_{H^{2}}\leq Cc_{0}\|\pi\|_{H^{3}}.\label{gradpH2}
   \end{eqnarray}
   Combining (\ref{u's 1}), (\ref{grad4u}),  (\ref{rouutH2}), (\ref{rouvgradu}) and (\ref{uH3})-(\ref{gradpH2}), one deduce
\be
\int_{0}^{t}\|u\|_{H^{4}}^{2}\mbox{d}t\leq Cc_{0}^{9+6\gamma}c_{1}^{5}c_{2}^{2},\label{uH4}
\ee
  for $0\leq t\leq T$.

Thus, we complete the proof of Lemma \ref{lem2}.
\end{proof}

In the following part, we estimate the turbulent kinetic energy  $k$.
\begin{lem}\label{lem3}
There exists a unique strong solution k to the initial boundary value problem (\ref{2.4}) and (\ref{1})  such that
\begin{eqnarray}
&&\|\sqrt{\rho}k_{t}\|_{L^{2}}^{2}
+\|k\|_{H^{1}}^{2}
+\int_{0}^{t}\|\nabla k_{t}\|_{L^{2}}^{2}\mbox{d}s\leq Cc_{0}^{5},\label{k1}\\
&&\|k\|_{H^{2}}\leq Cc_{0}^{\frac{7}{2}}c_{1}c_{2}^{2},\
 \int_{0}^{t}\|k\|_{H^{3}}^{2}\mbox{d}s\leq Cc_{0}^{7}\label{k2}
\end{eqnarray}
for $0\leq t\leq T$.
\end{lem}

\begin{proof}
We only need to prove the estimates.
Differentiating  equation (\ref{2.4}) with respect to t, then multiplying both sides of the result equation
by $k_{t}$ and integrating over $\Omega$, we get
\begin{eqnarray}
&&\frac{1}{2}\frac{\mbox{d}}{\mbox{d}t}\|\sqrt{\rho}k_{t}\|_{L^{2}}^{2}
+\|\nabla k_{t}\|_{L^{2}}^{2}\nonumber\\
&&=-\int \rho_{t}v\cdot\nabla k\cdot k_{t}
-\int \rho v_{t}\cdot\nabla k\cdot k_{t}
-2\int\rho v\cdot\nabla k_{t}\cdot k_{t}
+\int G_{t}^{'}\cdot k_{t}\nonumber\\
&&-\int \rho_{t}\theta\cdot k_{t}
-\int\rho\theta_{t}\cdot k_{t}
=\mathop{\sum}_{i=1}^{6}K_{i},\label{gradroukL2}
\end{eqnarray}
we could evaluate $K_{i}\ (i=1,\cdots,6)$ as follows.

First, using similar method of deriving $(\ref{I1})$, $(\ref{I2-2})$, $(\ref{I3})$, respectively, one has
\be
K_{1}
\leq Cc_{0}^{2}c_{2}^{4}\|\nabla k\|_{L^{2}}^{2}+C\|\sqrt{\rho}k_{t}\|_{L^{2}}^{2}+\frac{1}{10}\|\nabla k_{t}\|_{L^{2}}^{2},\label{K1}
\ee
\be
K_{2}\leq C\eta^{-1}c_{0}^{2\gamma+1}(\|\sqrt{\rho}k_{t}\|_{L^{2}}^{2}+c_{1}^{2}\|\nabla k\|_{L^{2}}^{2}
+c_{0}^{2}c_{1}^{2}c_{2}^{4})
+\eta\|v_{t}\|_{H^{1}}^{2}\|\sqrt{\rho}k_{t}\|_{L^{2}}^{2},\label{K2-2}
\ee

\be
K_{3}
\leq Cc_{0}c_{2}^{2}\|\sqrt{\rho}k_{t}\|_{L^{2}}^{2}+\frac{1}{10}\|\nabla k_{t}\|_{L^{2}}^{2}.\label{K3}
\ee
Next, differentiating $G^{'}$ with respect to $t$ and inserting the result thus obtained to $K_{4}$ yield
\begin{eqnarray}
&&K_{4}
\leq C\int|\nabla v_{t}||\nabla v||k_{t}|
+C\int|\rho||\pi||\nabla v_{t}||k_{t}|
+C\int|\rho_{t}||\pi||\nabla v||k_{t}|
+C\int|\rho||\pi_{t}||\nabla v||k_{t}|\nonumber\\
&&\leq Cc_{0}^{\frac{1}{2}}\|\sqrt{\rho}k_{t}\|_{L^{2}}\|\nabla v_{t}\|_{L^{2}}\|\nabla v\|_{L^{\infty}}
+Cc_{0}^{\frac{1}{2}}\|\pi\|_{L^{\infty}}\|\nabla v_{t}\|_{L^{2}}\|\sqrt{\rho}k_{t}\|_{L^{2}}\nonumber\\
&&+C\|\pi\|_{L^{\infty}}\|\rho_{t}\|_{L^{3}}\|\nabla v\|_{L^{2}}\|k_{t}\|_{L^{6}}
+Cc_{0}^{\frac{1}{2}}\|\sqrt{\rho}k_{t}\|_{L^{2}}\|\pi_{t}\|_{L^{6}}\|\nabla v\|_{L^{3}}\label{K4}\\
&&\leq C\eta^{-1}c_{0}c_{3}^{2}c_{5}^{2}+Cc_{0}^{2}c_{1}^{2}c_{2}^{2}c_{5}^{2}
+C\|\sqrt{\rho}k_{t}\|_{L^{2}}^{2}
+C\eta(\|v_{t}\|_{H^{1}}^{2}
+\|\pi_{t}\|_{H^{1}}^{2})\|\sqrt{\rho}k_{t}\|_{L^{2}}^{2}
+\frac{1}{10}\|\nabla k_{t}\|_{L^{2}}^{2}.\nonumber
\end{eqnarray}
Last, direct calculation leads to
\begin{eqnarray}
K_{5}\leq \|\rho_{t}\|_{L^{3}}\|\theta\|_{L^{2}}\|k_{t}\|_{L^{6}}
\leq Cc_{0}^{2}c_{1}^{2}c_{2}^{2}
+C\|\sqrt{\rho}k_{t}\|_{L^{2}}^{2}
+\frac{1}{10}\|\nabla k_{t}\|_{L^{2}}^{2},\label{K5}
\end{eqnarray}
\begin{eqnarray}
K_{6}\leq Cc_{0}^{\frac{1}{2}}\|\sqrt{\rho}k_{t}\|_{L^{2}}\|\theta_{t}\|_{L^{2}}
\leq C\eta^{-1}c_{0}+\eta\|\theta_{t}\|_{L^{2}}^{2}\|\sqrt{\rho}k_{t}\|_{L^{2}}^{2}.\label{K6}
\end{eqnarray}
On the other hand, we easily get
\begin{eqnarray}
&&\frac{\mbox{d}}{\mbox{d}t}\|\nabla k\|_{L^{2}}^{2}\leq \frac{1}{10}\|\nabla k_{t}\|_{L^{2}}^{2}
+C\|\nabla k\|_{L^{2}}^{2},\label{gradkL2}\\
&&\frac{\mbox{d}}{\mbox{d}t}\|k\|_{L^{2}}^{2}
\leq Cc_{0}\|\sqrt{\rho}k_{t}\|_{L^{2}}^{2}+C\|k\|_{L^{2}}^{2}.\label{kL2}
\end{eqnarray}
Combining (\ref{gradroukL2})-(\ref{kL2}), we obtain
\begin{eqnarray}
&&\frac{\mbox{d}}{\mbox{d}t}(\|\sqrt{\rho}k_{t}\|_{L^{2}}^{2}+\|k\|_{H^{1}}^{2})+\|\nabla k_{t}\|_{L^{2}}^{2}\nonumber\\
&&\leq C(c_{0}^{2}c_{2}^{4}+\eta^{-1}c_{0}^{2\gamma+1}c_{1}^{2}+\eta\|v_{t}\|_{H^{1}}^{2}+\eta\|\pi_{t}\|_{H^{1}}^{2}
+\eta\|\theta_{t}\|_{L^{2}}^{2})(\|\sqrt{\rho}k_{t}\|_{L^{2}}^{2}
+\|k\|_{H^{1}}^{2})\nonumber\\
&&+C(\eta^{-1}c_{0}^{2}c_{1}^{2}c_{2}^{4}c_{3}^{2}c_{5}^{2}+c_{0}^{2}c_{1}^{2}c_{2}^{2}c_{5}^{2}),
\end{eqnarray}
setting $\eta=c_{1}^{-1}$ and using Gronwall's inequality, we deduce
\be
\|\sqrt{\rho}k_{t}\|_{L^{2}}^{2}+\|k\|_{H^{1}}^{2}+\int_{0}^{t}\|\nabla k_{t}\|_{L^{2}}^{2}\mbox{d}s
\leq Cc_{0}^{5}\label{k's 1}
\ee
for $0\leq t\leq T$, where we have used the fact that $\mathop{\lim}_{t\rightarrow 0}(\|\sqrt{\rho}k_{t}\|_{L^{2}}^{2}+\|k\|_{H^{1}}^{2})\leq Cc_{0}^{5}$.

Then, by the standard elliptic regularity result of equation (\ref{2.4}) and using (\ref{k's 1}), we have
\begin{eqnarray}
\|\nabla k\|_{H^{1}}
\leq Cc_{0}^{\frac{1}{2}}\|\sqrt{\rho}k_{t}\|_{L^{2}}+Cc_{0}\|v\|_{L^{\infty}}\|\nabla k\|_{L^{2}}+
C\|\nabla v\|_{L^{4}}^{2}\nonumber\\
+Cc_{0}\|\pi\|_{L^{4}}\|\nabla v\|_{L^{4}}+Cc_{0}\|\theta\|_{L^{2}}
\leq Cc_{0}^{\frac{7}{2}}c_{1}c_{2}^{2},\label{k's 2}
\end{eqnarray}
and
\begin{eqnarray}
\|\nabla^{2} k\|_{H^{1}}\leq C(\|\rho k_{t}\|_{H^{1}}+\|\rho v\cdot\nabla k\|_{H^{1}}+\|G^{'}\|_{H^{1}}+\|\rho\theta\|_{H^{1}}).\label{grad2KH1}
\end{eqnarray}
To evaluate $\int_{0}^{t}\|k\|_{H^{3}}^{2}\mbox{d}t$,  we will estimate the right hand side of (\ref{grad2KH1}) item by item.

In fact, we derive by using (\ref{k's 1}) and (\ref{k's 2}) that
\begin{eqnarray}
\|\rho k_{t}\|_{H^{1}}\leq C(\|\rho k_{t}\|_{L^{2}}+\|\nabla(\rho k_{t})\|_{L^{2}})
\leq Cc_{0}^{\frac{7}{2}}+Cc_{0}\|\nabla k_{t}\|_{L^{2}},\label{rouktH1}
\end{eqnarray}
\begin{eqnarray}
&&\|\rho v\cdot\nabla k\|_{H^{1}}\leq C(\|\rho v\cdot\nabla k\|_{L^{2}}+\|\nabla(\rho v\cdot\nabla k)\|_{L^{2}})\nonumber\\
&&\leq C(c_{0}\|v\|_{L^{\infty}}\|\nabla k\|_{L^{2}}+\|\nabla\rho\|_{L^{\infty}}\|v\|_{L^{\infty}}\|\nabla k\|_{L^{2}}\nonumber\\
&&+c_{0}\|\nabla v\|_{L^{4}}\|\nabla k\|_{L^{4}}+c_{0}\|v\|_{L^{\infty}}\|\nabla^{2}k\|_{L^{2}})
\leq Cc_{0}^{\frac{9}{2}}c_{1}c_{2}^{3},\label{rouvgradkH1}
\end{eqnarray}
\begin{eqnarray}
&&\|G^{'}\|_{H^{1}}\leq C(\|\nabla v\|_{L^{4}}^{2}+\|\nabla v\cdot\rho\cdot\pi\|_{L^{2}}+\|\nabla v\cdot\nabla^{2}v\|_{L^{2}}
+\|\nabla(\nabla v\cdot\rho\cdot\pi)\|_{L^{2}})\nonumber\\
&&\leq C(\|\nabla v\|_{L^{4}}^{2}+c_{0}\|\pi\|_{L^{\infty}}\|\nabla v\|_{L^{2}}+\|\nabla v\|_{L^{4}}\|\nabla^{2} v\|_{L^{4}}
+c_{0}\|\pi\|_{L^{\infty}}\|\nabla^{2} v\|_{L^{2}}\nonumber\\
&&+\|\pi\|_{L^{\infty}}\|\nabla\rho\|_{L^{\infty}}\|\nabla v\|_{L^{2}}+c_{0}\|\nabla v\|_{L^{\infty}}\|\nabla\pi\|_{L^{2}})
\leq Cc_{0}c_{1}c_{2}^{2}c_{3}c_{5},\label{GH1}
\end{eqnarray}
and
\begin{eqnarray}
\|\rho\theta\|_{H^{1}}\leq C\|\rho\|_{H^{3}}\|\theta\|_{H^{1}}\leq Cc_{0}c_{1}.\label{routhetaH1}
\end{eqnarray}
Therefore, inserting (\ref{rouktH1})-(\ref{routhetaH1}) to (\ref{grad2KH1}) and  integrating the result thus obtained over $(0,t)$, one gets
\be
\int_{0}^{t}\|k\|_{H^{3}}^{2}\mbox{d}t\leq Cc_{0}^{7}\label{k's 3}
\ee
for $0\leq t\leq T$.

Combining (\ref{k's 1}), (\ref{k's 2}) and (\ref{k's 3}), we complete the proof of Lemma \ref{lem3}.
\end{proof}

In the next part, we estimate the viscous dissipation rates of the turbulent flows $\varepsilon$.
\begin{lem}\label{lem4}
There exists a unique strong solution $\varepsilon$ to the initial boundary value problem (\ref{2.5}) and  (\ref{1})  such that
\begin{eqnarray}
&&\|\sqrt{\rho}\varepsilon_{t}\|_{L^{2}}^{2}
+\|\varepsilon\|_{H^{1}}^{2}
+\int_{0}^{t}\|\nabla \varepsilon_{t}\|_{L^{2}}^{2}\mbox{d}s\leq Cc_{0}^{5},\label{e1}\\
&&\|\varepsilon\|_{H^{2}}\leq Cc_{0}^{\frac{9}{2}}c_{1}^{2}c_{2}^{2}\label{e2}
\end{eqnarray}
for $0\leq t\leq T$.
\end{lem}

\begin{proof}
We only need to prove the estimates.
Differentiating equation (\ref{2.5}) with respect to t, then multiplying both sides of the result
by $\varepsilon_{t}$ and integrating over $\Omega$, one obtains
\begin{eqnarray}
&&\frac{1}{2}\frac{\mbox{d}}{\mbox{d}t}\|\sqrt{\rho}\varepsilon_{t}\|_{L^{2}}^{2}
+\|\nabla \varepsilon_{t}\|_{L^{2}}^{2}\nonumber\\
&&= -\int\rho_{t}v\cdot\nabla\varepsilon\cdot\varepsilon_{t}
-\int\rho v_{t}\cdot\nabla\varepsilon\cdot\varepsilon_{t}
-2\int\rho v\cdot\nabla\varepsilon_{t}\cdot\varepsilon_{t}\nonumber\\
&&+\int\bigg(\frac{C_{1}G^{'}\theta}{\pi}\bigg)_{t}\cdot\varepsilon_{t}
-\int\bigg(\frac{C_{2}\rho\theta^{2}}{\pi}\bigg)_{t}\cdot\varepsilon_{t}
=\mathop{\sum}_{i=1}^{5}E_{i}.\label{e eqn}
\end{eqnarray}
We could evaluate $E_{4}$ and $E_{5}$ in the first place. Because $\pi$ has upper and lower bound away from zero, direct calculation yields
\begin{eqnarray}
&&E_{4}\leq C\int(|G^{'}_{t}\theta|+|G^{'}\theta_{t}|+|G^{'}\theta\pi_{t}|)|\varepsilon_{t}|\nonumber\\
&&\leq C\int(|\nabla v_{t}\cdot\nabla v|+|\rho_{t}\pi\nabla v|+|\rho\pi_{t}\nabla v|+|\rho\pi\nabla v_{t}|)|\theta||\varepsilon_{t}|\nonumber\\
&&+C\int(|\nabla v|^{2}+|\rho\pi\nabla v|)|\theta_{t}||\varepsilon_{t}|
+C\int(|\nabla v|^{2}+|\rho\pi\nabla v|)|\theta||\pi_{t}||\varepsilon_{t}|\nonumber\\
&&\leq Cc_{0}^{\frac{1}{2}}\|\theta\|_{L^{\infty}}\|\nabla v\|_{L^{\infty}}\|\nabla v_{t}\|_{L^{2}}\|\sqrt{\rho}\varepsilon_{t}\|_{L^{2}}
+Cc_{0}^{\frac{1}{2}}\|\pi\|_{L^{\infty}}\|\sqrt{\rho}\varepsilon_{t}\|_{L^{2}}\|\rho_{t}\|_{L^{6}}\|\nabla v\|_{L^{6}}\|\theta\|_{L^{6}}\nonumber\\
&&+Cc_{0}^{\frac{1}{2}}\|\sqrt{\rho}\varepsilon_{t}\|_{L^{2}}\|\pi_{t}\|_{L^{6}}\|\nabla v\|_{L^{6}}\|\theta\|_{L^{6}}
+Cc_{0}\|\pi\|_{L^{\infty}}\|\theta\|_{L^{\infty}}\|\nabla v_{t}\|_{L^{2}}\|\sqrt{\rho}\varepsilon_{t}\|_{L^{2}}\nonumber\\
&&+C\|\sqrt{\rho}\varepsilon_{t}\|_{L^{2}}\|\theta_{t}\|_{L^{2}}\|\nabla v\|_{L^{\infty}}^{2}
+Cc_{0}\|\pi\|_{L^{\infty}}\|\sqrt{\rho}\varepsilon_{t}\|_{L^{2}}\|\theta_{t}\|_{L^{2}}\|\nabla v\|_{L^{\infty}}\nonumber\\
&&+Cc_{0}^{\frac{1}{2}}\|\sqrt{\rho}\varepsilon_{t}\|_{L^{2}}\|\pi_{t}\|_{L^{6}}\|\nabla v\|_{L^{6}}^{2}\|\theta\|_{L^{\infty}}
+Cc_{0}^{\frac{1}{2}}\|\pi\|_{L^{\infty}}\|\sqrt{\rho}\varepsilon_{t}\|_{L^{2}}\|\pi_{t}\|_{L^{6}}\|\nabla v\|_{L^{6}}\|\theta\|_{L^{6}}\label{E4}\\
&&\leq C\eta^{-1}c_{0}c_{1}^{2}c_{2}^{4}c_{6}^{2}c_{3}^{4}c_{5}^{2}+Cc_{0}^{4}c_{1}^{2}c_{2}^{4}c_{5}^{2}
+C\eta(\|\nabla v_{t}\|_{L^{2}}^{2}+\|\pi_{t}\|_{L^{6}}^{2}+\|\theta_{t}\|_{L^{2}}^{2})
\|\sqrt{\rho}\varepsilon_{t}\|_{L^{2}}^{2}+C\|\sqrt{\rho}\varepsilon_{t}\|_{L^{2}}^{2},\nonumber
\end{eqnarray}
and
\begin{eqnarray}
&&E_{5}\leq C\int|\rho_{t}\theta^{2}\varepsilon_{t}|
+C\int|\theta\theta_{t}\rho\varepsilon_{t}|
+C\int|\rho\theta^{2}\pi_{t}\varepsilon_{t}|\nonumber\\
&&\leq C\|\rho_{t}\|_{L^{3}}\|\theta\|_{L^{4}}^{2}\|\varepsilon_{t}\|_{L^{6}}
+Cc_{0}^{\frac{1}{2}}\|\sqrt{\rho}\varepsilon_{t}\|_{L^{2}}\|\theta_{t}\|_{L^{2}}\|\theta\|_{L^{\infty}}
+Cc_{0}^{\frac{1}{2}}\|\sqrt{\rho}\varepsilon_{t}\|_{L^{2}}\|\pi_{t}\|_{L^{2}}\|\theta\|_{L^{\infty}}^{2}\label{E5}\\
&&\leq C\eta^{-1}c_{0}c_{6}^{4}+Cc_{0}^{2}c_{1}^{4}c_{2}^{2}+C\|\sqrt{\rho}\varepsilon_{t}\|_{L^{2}}^{2}+C\eta(\|\theta_{t}\|_{L^{2}}^{2}
+\|\pi_{t}\|_{L^{2}}^{2})\|\sqrt{\rho}\varepsilon_{t}\|_{L^{2}}^{2}+\frac{1}{8}\|\nabla\varepsilon_{t}\|_{L^{2}}^{2}.\nonumber
\end{eqnarray}

Next, using an argument similar to that used in deriving (\ref{K1}), (\ref{K2-2}), (\ref{K3}), (\ref{kL2}) and (\ref{gradkL2}), respectively, one gets
\be
E_{1}
\leq Cc_{0}^{2}c_{2}^{4}\|\nabla \varepsilon\|_{L^{2}}^{2}+C\|\sqrt{\rho}\varepsilon_{t}\|_{L^{2}}^{2}+\frac{1}{10}\|\nabla \varepsilon_{t}\|_{L^{2}}^{2},\label{E1}
\ee
\be
E_{2}\leq C\eta^{-1}c_{0}^{2\gamma+1}(\|\sqrt{\rho}\varepsilon_{t}\|_{L^{2}}^{2}+c_{1}^{2}\|\nabla \varepsilon\|_{L^{2}}^{2}
+c_{1}^{4}c_{2}^{4})
+\eta\|v_{t}\|_{H^{1}}^{2}\|\sqrt{\rho}\varepsilon_{t}\|_{L^{2}}^{2},\label{E2-2}
\ee

\be
E_{3}
\leq Cc_{0}c_{2}^{2}\|\sqrt{\rho}\varepsilon_{t}\|_{L^{2}}^{2}+\frac{1}{10}\|\nabla \varepsilon_{t}\|_{L^{2}}^{2},\label{E3}
\ee
\be
\frac{\mbox{d}}{\mbox{d}t}\|\varepsilon\|_{L^{2}}^{2}
\leq  C\|\varepsilon\|_{L^{2}}^{2}+Cc_{0}\|\sqrt{\rho}\varepsilon_{t}\|_{L^{2}}^{2},\label{eL2}
\ee
and finally
\be
\frac{\mbox{d}}{\mbox{d}t}\|\nabla\varepsilon\|_{L^{2}}^{2}
\leq \frac{1}{8}\|\nabla\varepsilon_{t}\|_{L^{2}}^{2}
+C\|\nabla\varepsilon\|_{L^{2}}^{2}.\label{gradeL2}
\ee

 Combining (\ref{e eqn})-(\ref{gradeL2}), one obtains
\begin{eqnarray}
&&\frac{1}{2}\frac{\mbox{d}}{\mbox{d}t}(\|\sqrt{\rho}\varepsilon_{t}\|_{L^{2}}^{2}+
\|\varepsilon\|_{H^{1}}^{2})
+\|\nabla \varepsilon_{t}\|_{L^{2}}^{2}\nonumber\\
&&\leq C(c_{0}^{2}c_{2}^{4}
+\eta^{-1}c_{0}^{2\gamma+1}c_{1}^{2}
+\eta\|v_{t}\|_{H^{1}}^{2}
+\eta\|\theta_{t}\|_{H^{1}}^{2}
+\eta\|\pi_{t}\|_{H^{1}}^{2}
)(\|\sqrt{\rho}\varepsilon_{t}\|_{L^{2}}^{2}+
\|\varepsilon\|_{H^{1}}^{2})\nonumber\\
&&+C\eta^{-1}c_{0}c_{1}^{4}c_{2}^{4}c_{6}^{4}c_{3}^{4}c_{5}^{2}
+Cc_{0}^{4}c_{1}^{4}c_{2}^{4}c_{5}^{2},
\end{eqnarray}
setting $\eta=c_{1}^{-1}$ and using Gronwall's inequality, one obtains
\begin{eqnarray}
\|\sqrt{\rho}\varepsilon_{t}\|_{L^{2}}^{2}
+\|\varepsilon\|_{H^{1}}^{2}
+\int_{0}^{t}\|\nabla \varepsilon_{t}\|_{L^{2}}^{2}\mbox{d}s\leq Cc_{0}^{5}\label{e's 1}
\end{eqnarray}
for $0\leq t\leq T$, where we have used the fact that $\mathop{\lim}_{t\rightarrow 0}(\|\sqrt{\rho}\varepsilon_{t}\|_{L^{2}}^{2}+\|\varepsilon\|_{H^{1}}^{2})\leq Cc_{0}^{5}$.

Next, applying the standard elliptic regularity result to equation (\ref{2.5}) and using (\ref{e's 1}), we have
\begin{eqnarray}
&&\|\nabla\varepsilon\|_{H^{1}}
\leq C(c_{0}^{\frac{1}{2}}\|\sqrt{\rho}\varepsilon_{t}\|_{L^{2}}
+c_{0}\|v\|_{L^{6}}\|\nabla\varepsilon\|_{L^{3}}
+\|\nabla v\|_{L^{6}}^{2}\|\theta\|_{L^{6}}+c_{0}\|\nabla v\|_{L^{6}}\|\theta\|_{L^{6}}\|\pi\|_{L^{6}}
+c_{0}\|\theta\|_{L^{4}}^{2})\nonumber\\
&&\leq C(c_{0}^{3}c_{2}^{2}c_{1}^{2}+c_{0}c_{1}\|\nabla\varepsilon\|_{L^{2}}^{\frac{1}{2}}\|\nabla\varepsilon\|_{L^{6}}^{\frac{1}{2}}),\label{gradeH1-3}
\end{eqnarray}
therefore, by Young's inequality and (\ref{e's 1}), one deduce \be
\|\varepsilon\|_{H^{2}}\leq
Cc_{0}^{\frac{9}{2}}c_{1}^{2}c_{2}^{2}.\nonumber \ee Thus, we
complete the proof of Lemma \ref{lem4}.
\end{proof}

Finally, we estimate the total enthalpy h.
\begin{lem}\label{lem5}
There exists a unique strong solution h to the initial boundary value problem(\ref{2.3}) and  (\ref{1})   such that
\begin{eqnarray}
&&\|\sqrt{\rho}h_{t}\|_{L^{2}}^{2}
+\|h\|_{H^{1}}^{2}
+\int_{0}^{t}\|\nabla h_{t}\|_{L^{2}}^{2}\mbox{d}s\leq Cc_{0}^{5},\label{h1}\\
&&\|h\|_{H^{2}}\leq Cc_{0}^{\frac{7}{2}+\gamma}c_{1}^{2}c_{2}^{2}\label{h2}
\end{eqnarray}
for $0\leq t\leq T$.
\end{lem}
\begin{proof}
We only need to prove the estimates.
Differentiating  equation (\ref{2.3}) with respect to t, multiplying both sides of the result equation
by $h_{t}$ and integrating over $\Omega$, one obtains
\begin{eqnarray}
&&\frac{\mbox{d}}{\mbox{d}t}(\|\sqrt{\rho}h_{t}\|_{L^{2}}^{2}+\|h\|_{H^{1}}^{2})
+\|\nabla h_{t}\|_{L^{2}}^{2}\nonumber\\
&&=-\int\rho_{t}v\cdot\nabla h\cdot h_{t}
-\int \rho v_{t}\cdot\nabla h\cdot h_{t}
-2\int\rho v\cdot\nabla h_{t}\cdot h_{t}
+\int p_{tt}\cdot h_{t}\nonumber\\
&&+\int u_{t}\cdot\nabla p\cdot h_{t}
+\int u\cdot\nabla p_{t}\cdot h_{t}
+\int S_{kt}^{'}\cdot h_{t}
=\mathop{\sum}_{i=1}^{7}H_{i}.\label{rouhtL2}
\end{eqnarray}
Firstly, using similar method of deriving the estimates (\ref{I1}), (\ref{I2-2}) and (\ref{I3}), respectively,  one has
\begin{eqnarray}
H_{1}
\leq Cc_{0}^{2}c_{2}^{4}\|\nabla h\|_{L^{2}}^{2}+C\|\sqrt{\rho}h_{t}\|_{L^{2}}^{2}+\frac{1}{20}\|\nabla h_{t}\|_{L^{2}},\label{H1}
\end{eqnarray}
\begin{eqnarray}
H_{2}\leq C\eta^{-1}c_{0}^{2\gamma+1}(c_{0}^{7}c_{2}^{4}+\|\sqrt{\rho}h_{t}\|_{L^{2}}^{2}+c_{1}^{2}\|\nabla h\|_{L^{2}}^{2})
+\eta\|v_{t}\|_{H^{1}}^{2}\|\sqrt{\rho}h_{t}\|_{L^{2}}^{2},\label{H2-2}
\end{eqnarray}
\begin{eqnarray}
H_{3}
\leq Cc_{0}c_{2}^{2}\|\sqrt{\rho}h_{t}\|_{L^{2}}^{2}+\frac{1}{20}\|\nabla h_{t}\|_{L^{2}}.\label{H3}
\end{eqnarray}
Secondly, differentiating equation (\ref{2.1}) with respect to $t$ yields
\begin{eqnarray}
\rho_{tt}=-\rho_{t}\nabla\cdot v
+\rho\nabla\cdot v_{t}
+v_{t}\cdot\nabla\rho
+v\cdot\nabla\rho_{t}.\label{routt}
\end{eqnarray}
Therefore, by direct calculation and using (\ref{routt}), we derive
\begin{eqnarray}
&&H_{4}
=\int [\gamma(\gamma-1)\rho^{\gamma-2}\rho_{t}^{2}
-\gamma\rho^{\gamma-1}(\rho_{t}\nabla\cdot v
+\rho\nabla\cdot v_{t}
+v_{t}\cdot\nabla\rho
+v\cdot\nabla\rho_{t})]\cdot h_{t}\nonumber\\
&&\leq Cc_{0}^{\gamma-\frac{3}{2}}\|\rho_{t}\|_{L^{4}}^{2}\|\sqrt{\rho}h_{t}\|_{L^{2}}
+Cc_{0}^{\gamma-\frac{1}{2}}\|\rho_{t}\|_{L^{3}}\|\nabla v\|_{L^{6}}\|\sqrt{\rho}h_{t}\|_{L^{2}}
+Cc_{0}^{\gamma-\frac{1}{2}}\|\sqrt{\rho}h_{t}\|_{L^{2}}\|\nabla v_{t}\|_{L^{2}}\nonumber\\
&&+Cc_{0}^{\gamma-\frac{1}{2}}\|\sqrt{\rho}h_{t}\|_{L^{2}}\|v_{t}\|_{L^{6}}\|\nabla\rho\|_{L^{3}}
+Cc_{0}^{\gamma-\frac{1}{2}}\|\nabla\rho_{t}\|_{L^{2}}\|v\|_{L^{\infty}}\|\sqrt{\rho}h_{t}\|_{L^{2}}\nonumber\\
&&\leq C(c_{0}^{2\gamma+1}c_{2}^{4}+\eta^{-1}c_{0}^{2\gamma}+\|\sqrt{\rho}h_{t}\|_{L^{2}}^{2}
+\eta\|v_{t}\|_{H^{1}}^{2}\|\sqrt{\rho}h_{t}\|_{L^{2}}^{2})+\frac{1}{20}\|\nabla h_{t}\|_{L^{2}}^{2}.\label{H4}
\end{eqnarray}

Thirdly, simple calculation and (\ref{u's 1}) lead to
\begin{eqnarray}
&&H_{5}\leq Cc_{0}^{\gamma-\frac{1}{2}}\|\sqrt{\rho}h_{t}\|_{L^{2}}\|\nabla\rho\|_{L^{3}}\|u_{t}\|_{L^{6}}
\leq Cc_{0}^{\gamma-\frac{1}{2}}\|\sqrt{\rho}h_{t}\|_{L^{2}}\|\nabla\rho\|_{L^{3}}(\|u_{t}\|_{L^{2}}+\|\nabla u_{t}\|_{L^{2}})\nonumber\\
&&\leq Cc_{0}^{2\gamma+1}\|\sqrt{\rho}h_{t}\|_{L^{2}}^{2}+Cc_{0}^{2\gamma+5}
+C\|\nabla u_{t}\|_{L^{2}}^{2}.\label{H5}
\end{eqnarray}

Next, by direct calculation, we know that $\nabla p_{t}= \gamma(\gamma-1)\rho^{\gamma-2}\rho_{t}\nabla\rho
+\gamma\rho^{\gamma-1}\nabla\rho_{t}$. Therefore,
\begin{eqnarray}
&&H_{6}\leq Cc_{0}^{\gamma-2}\int|\rho_{t}||u||\nabla\rho|| h_{t}|+Cc_{0}^{\gamma-1}\int |u||\nabla\rho_{t}||h_{t}|\nonumber\\
&&\leq Cc_{0}^{\gamma-2}\|\nabla\rho\|_{L^{\infty}}\|\rho_{t}\|_{L^{3}}\|u\|_{L^{2}}(\|\sqrt{\rho}h_{t}\|_{L^{2}}
+\|\nabla h_{t}\|_{L^{2}})\nonumber\\
&&+Cc_{0}^{\gamma-1}\|u\|_{L^{3}}\|\nabla\rho_{t}\|_{L^{2}}(\|\sqrt{\rho}h_{t}\|_{L^{2}}+\|\nabla h_{t}\|_{L^{2}})\nonumber\\
&&\leq Cc_{0}^{7+2\gamma}c_{2}^{2}+C\|\sqrt{\rho}h_{t}\|_{L^{2}}^{2}+\frac{1}{20}\|\nabla h_{t}\|_{L^{2}}^{2}.\label{H6}
\end{eqnarray}

Last, simple calculation yields $|S_{kt}^{'}|\leq C|\nabla v||\nabla v_{t}|+C\rho^{\gamma-1}|\rho_{t}||\nabla\rho|^{2}
+C\rho^{\gamma-1}|\nabla\rho_{t}||\nabla\rho|$, thus
\begin{eqnarray}
&&H_{7}
\leq C\int|\nabla v_{t}||\nabla v||h_{t}|+Cc_{0}^{\gamma-1}\int|\rho_{t}||\nabla\rho|^{2}|h_{t}|
+Cc_{0}^{\gamma-1}\int|\nabla\rho_{t}||\nabla\rho||h_{t}|\nonumber\\
&&\leq Cc_{0}^{\frac{1}{2}}\|\nabla v\|_{L^{\infty}}\|\nabla v_{t}\|_{L^{2}}\|\sqrt{\rho}h_{t}\|_{L^{2}}
+Cc_{0}^{\gamma-\frac{1}{2}}\|\rho_{t}\|_{L^{6}}\|\nabla\rho\|_{L^{6}}^{2}\|\sqrt{\rho}h_{t}\|_{L^{2}}\nonumber\\
&&+Cc_{0}^{\gamma-\frac{1}{2}}\|\nabla\rho\|_{L^{\infty}}\|\nabla\rho_{t}\|_{L^{2}}\|\sqrt{\rho}h_{t}\|_{L^{2}}\nonumber\\
&&\leq C(\eta^{-1}c_{0}c_{3}^{2}+c_{0}^{5+2\gamma}c_{2}^{2}
+\eta\|\nabla v_{t}\|_{L^{2}}^{2}\|\sqrt{\rho}h_{t}\|_{L^{2}}^{2}
+\|\sqrt{\rho}h_{t}\|_{L^{2}}^{2}).\label{H7}
\end{eqnarray}
Furthermore, we easily have
\be
\frac{\mbox{d}}{\mbox{d}t}\|h\|_{L^{2}}^{2}\leq Cc_{0}\|\sqrt{\rho}h_{t}\|_{L^{2}}^{2}+C\|h\|_{L^{2}}^{2},\label{hL2}
\ee
and
\be
\frac{\mbox{d}}{\mbox{d}t}\|\nabla h\|_{L^{2}}^{2}\leq C\|\nabla h\|_{L^{2}}^{2}+\frac{1}{10}\|\nabla h_{t}\|_{L^{2}}^{2}.\label{gradhL2}
\ee
Consequently, combining (\ref{rouhtL2})-(\ref{gradhL2}), one deduces
\begin{eqnarray}
&&\frac{\mbox{d}}{\mbox{d}t}(\|\sqrt{\rho}h_{t}\|_{L^{2}}^{2}+\|h\|_{H^{1}}^{2})
+\|\nabla h_{t}\|_{L^{2}}^{2}\nonumber\\
&&\leq C(c_{0}^{2\gamma+1}c_{2}^{4}+\eta^{-1}c_{0}^{2\gamma+1}c_{1}^{2}+\eta\|v_{t}\|_{H^{1}}^{2})(\|\sqrt{\rho}h_{t}\|_{L^{2}}^{2}+\|h\|_{H^{1}}^{2})\nonumber\\
&&+C(c_{0}^{7+2\gamma}c_{2}^{4}+\eta^{-1}c_{0}^{8+2\gamma}c_{2}^{4}c_{3}^{2}),
\end{eqnarray}
setting $\eta=c_{1}^{-1}$ and using Gronwall's inequality, we get
\begin{eqnarray}
\|\sqrt{\rho}h_{t}\|_{L^{2}}^{2}+\|h\|_{H^{1}}^{2}
+\int_{0}^{t}\|\nabla h_{t}\|_{L^{2}}^{2}\mbox{d}s
\leq Cc_{0}^{5}\label{h's 1}
\end{eqnarray}
for $0\leq t \leq T$, where we have used the fact that $\mathop{\lim}_{t\rightarrow 0}(\|\sqrt{\rho}h_{t}\|_{L^{2}}^{2}+\|h\|_{H^{1}}^{2})\leq Cc_{0}^{5}$.

Next, using (\ref{h's 1}) and the standard elliptic regularity result of equation (\ref{2.3}), one obtains
\begin{eqnarray}
&&\|\nabla h\|_{H^{1}}
\leq C(c_{0}^{\frac{1}{2}}\|\sqrt{\rho}h_{t}\|_{L^{2}}
+c_{0}\|v\|_{L^{6}}\|\nabla h\|_{L^{3}}
+c_{0}^{\gamma-1}\|\rho_{t}\|_{L^{2}}+c_{0}^{\gamma-1}\|u\|_{L^{6}}\|\nabla \rho\|_{L^{3}}\nonumber\\
&&+\|\nabla v\|_{L^{4}}^{2}+c_{0}^{\gamma-1}\|\nabla\rho\|_{L^{4}}^{2})
\leq Cc_{0}^{\frac{5}{2}+\gamma}c_{2}^{2}+Cc_{0}c_{1}\|\nabla h\|_{L^{2}}^{\frac{1}{2}}\|\nabla h\|_{H^{1}}^{\frac{1}{2}},
\end{eqnarray}
then, Young's inequality and (\ref{h's 1}) yields
\be
\|h\|_{H^{2}}\leq Cc_{0}^{\frac{7}{2}+\gamma}c_{1}^{2}c_{2}^{2}.\nonumber
\ee
Thus, we have finished the proof of Lemma \ref{lem5}.
\end{proof}
Next, let us define $c_{i}\ (i=1,\cdots,6)$ as follows:
\begin{eqnarray}
c_{1}=Cc_{0}^{7+2\gamma},
c_{2}=Cc_{0}^{\frac{5}{2}+3\gamma}c_{1}^{2},
c_{5}=Cc_{0}^{\frac{7}{2}}c_{1}c_{2}^{2},
c_{6}=Cc_{0}^{\frac{9}{2}}c_{1}^{2}c_{2}^{2},
c_{3}=Cc_{0}^{\frac{13}{2}+3\gamma}c_{1}^{4}c_{2}c_{5},
c_{4}=Cc_{0}^{9+6\gamma}c_{1}^{5}c_{2}^{2},\nonumber
\end{eqnarray}
then we conclude from Lemma \ref{lem1} to Lemma \ref{lem5} that
\begin{eqnarray}\begin{cases}\label{conclusion1}
 \mathop{\sup}_{0\leq t\leq T}(\|u\|_{H^{1}}
 +\|k\|_{H^{1}}+\|\varepsilon\|_{H^{1}})\\
 +\int_{0}^{T}(\|k\|_{H^{3}}^{2}
 +\|u_{t}\|_{H^{1}}^{2}+\|k_{t}\|_{H^{1}}^{2}+\|\varepsilon_{t}\|_{H^{1}}^{2})\mbox{d}t\leq c_{1},\\
 \mathop{\sup}_{0\leq t\leq T}\|u\|_{H^{2}}\leq c_{2},
 \mathop{\sup}_{0\leq t\leq T}\|u\|_{H^{3}}\leq c_{3},
  \int_{0}^{T}\|u\|_{H^{4}}^{2}\mbox{d}t\leq c_{4},\\
   \mathop{\sup}_{0\leq t\leq T}\|k\|_{H^{2}}\leq c_{5},
   \mathop{\sup}_{0\leq t\leq T}\|\varepsilon\|_{H^{2}}\leq c_{6}
   \end{cases}
\end{eqnarray}
and
\begin{eqnarray}
 \begin{cases}\label{conclusion2}
\|\rho\|_{H^{3}(\Omega)}\leq Cc_{0},\  \|\rho_{t}\|_{H^{1}(\Omega)}\leq Cc_{0}c_{2}\\
\|\sqrt{\rho}h_{t}\|_{L^{2}}^{2}
+\|h\|_{H^{1}}^{2}
+\int_{0}^{t}\|\nabla h_{t}\|_{L^{2}}^{2}\mbox{d}s\leq Cc_{0}^{5},\\
\|h\|_{H^{2}}\leq Cc_{0}^{\frac{7}{2}+\gamma}c_{1}^{2}c_{2}^{2}
 \end{cases}
\end{eqnarray}
for $0\leq t\leq T$.

Using standard proof as that in \cite{C-K3}, we can complete the proof of Theorem \ref{lem6}. \qed


 \section{Existence of strong solutions to the $k-\varepsilon$ equations}
 \setcounter{equation}{0}

\begin{thm}
There exists a small time $T^{*}>0$ and a unique strong solution $(\rho,u,h,k,\varepsilon) $
to the initial boundary value problem (\ref{1.1})-(\ref{1.9}) such that
\begin{eqnarray}
&&\rho \in C(0,T^{*};H^{3}),
\rho_{t} \in C(0,T^{*};H^{1}),
u \in C(0,T^{*}; H^{3})\cap L^{2}(0,T^{*}; H^{4}),\nonumber\\
&&u_{t}\in L^{2}(0,T^{*};H^{1}),
k\in C(0,T^{*}; H^{2})\cap L^{2}(0,T^{*};H^{3}),
k_{t}\in L^{2}(0,T^{*};H^{1}),\nonumber\\
&&\varepsilon\in C(0,T^{*}; H^{2}),
\varepsilon_{t}\in L^{2}(0,T^{*};H^{1}),
h\in C(0,T^{*}; H^{2}),
h_{t}\in L^{2}(0,T^{*};H^{1}),\nonumber\\
&&(\sqrt{\rho}u_{t},\sqrt{\rho}k_{t},\sqrt{\rho}\varepsilon_{t},\sqrt{\rho}h_{t})\in L^{\infty}(0,T^{*};L^{2}).\label{regularity}
\end{eqnarray}
\end{thm}
\begin{proof}
Our proof will be based on the iteration argument and on the results
in the last section (especially Theorem \ref{lem6}).

Firstly, using the regularity effect of classical heat equation, we can construct functions $(u^{0}=u^{0}(x,t), k^{0}=k^{0}(x,t), \varepsilon^{0}=\varepsilon^{0}(x,t))$ satisfying
$(u^{0}(x,0), k^{0}(x,0), \varepsilon^{0}(x,0))
=(u_{0}(x), k_{0}(x), \varepsilon_{0}(x))$ and
\begin{eqnarray}\begin{cases}
\mathop{\sup}_{0\leq t\leq T}(\|u^{0}\|_{H^{1}}
 +\|k^{0}\|_{H^{1}}+\|\varepsilon^{0}\|_{H^{1}})\\
 +\int_{0}^{T}(\|k^{0}\|_{H^{3}}^{2}
 +\|u^{0}_{t}\|_{H^{1}}^{2}+\|k^{0}_{t}\|_{H^{1}}^{2}+\|\varepsilon^{0}_{t}\|_{H^{1}}^{2})\mbox{d}t\leq c_{1},\nonumber\\
\mathop{\sup}_{0\leq t\leq T}\|u^{0}\|_{H^{2}}\leq c_{2},   \
 \mathop{\sup}_{0\leq t\leq T}\|u^{0}\|_{H^{3}}\leq c_{3},   \
  \int_{0}^{T}\|u^{0}\|_{H^{4}}^{2}\mbox{d}t\leq c_{4},\nonumber\\
   \mathop{\sup}_{0\leq t\leq T}\|k^{0}\|_{H^{2}}\leq c_{5},  \
   \mathop{\sup}_{0\leq t\leq T}\|\varepsilon^{0}\|_{H^{2}}\leq c_{6}.
   \end{cases}
\end{eqnarray}
Therefore it follows from Theorem \ref{lem6} that there exists a
unique strong solution
$(\rho^{1},u^{1},h^{1},k^{1},\varepsilon^{1})$ to the linearized
problem (\ref{2.1})-(\ref{2.5}) with $v, \pi, \theta$ replaced by
$u^{0}, k^{0}, \varepsilon^{0}$, respectively, which satisfies the
regularity estimates (\ref{conclusion1}) and (\ref{conclusion2}).
Similarly, we construct approximate solutions
$(\rho^{n},u^{n},h^{n},k^{n},\varepsilon^{n})$, inductively, as
follows: assuming that $u^{n-1}, k^{n-1},\varepsilon^{n-1}$ have
been defined for $n\geq 1$, let
$(\rho^{n},u^{n},h^{n},k^{n},\varepsilon^{n})$ be the unique
solution to the linearized problem (\ref{2.1})-(\ref{2.5}) with $v,
\pi, \theta$ replaced by $u^{n-1}, k^{n-1}, \varepsilon^{n-1}$,
respectively. Then it follows from Theorem \ref{lem6} that there exists a
constant $\widetilde{C}>1$ such that
\begin{eqnarray}
&& \sup_{0\leq t\leq T}(\|\rho^{n}\|_{H^{3}}+\|\rho^{n}_{t}\|_{H^{1}})
 +\sup_{0\leq t\leq T}(\|u^{n}\|_{H^{3}}+\|k^{n}\|_{H^{2}}+\|\varepsilon^{n}\|_{H^{2}}
 +\|h^{n}\|_{H^{2}})\nonumber\\
 &&+\sup_{0\leq t\leq T}(\|\sqrt{\rho^{n}}u^{n}_{t}\|_{L^{2}}+
 \|\sqrt{\rho^{n}}h^{n}_{t}\|_{L^{2}}+\|\sqrt{\rho^{n}}k^{n}_{t}\|_{L^{2}}+\|\sqrt{\rho^{n}}\varepsilon^{n}_{t}\|_{L^{2}})\nonumber\\
 &&+\int_{0}^{T}(\|u^{n}_{t}\|_{H^{1}}^{2}
 +\|h^{n}_{t}\|_{H^{1}}^{2}
 +\|k^{n}_{t}\|_{H^{1}}^{2}
 +\|\varepsilon^{n}_{t}\|_{H^{1}}^{2}
 +\|u^{n}\|_{H^{4}}^{2}
 +\|k^{n}\|_{H^{3}}^{2})\leq \widetilde{C}\label{another conclusion}
\end{eqnarray}
for all $n\geq 1$. Throughout the proof, we denote by $\widetilde{C}$ a generic constant
depending only on $m$,  $\gamma$, $|\Omega|$ and $c_{0}$, but independent of n.
Next, we will show that the full sequence $(\rho^{n},u^{n},h^{n},k^{n},\varepsilon^{n})$
converges to a solution to the original nonlinear problem (\ref{1.1})-(\ref{1.9}) in the strong sense.

Define $\overline{\rho}^{n+1}=\rho^{n+1}-\rho^{n}$, $\overline{u}^{n+1}=u^{n+1}-u^{n}$,\
$\overline{h}^{n+1}=h^{n+1}-h^{n}$, $\overline{k}^{n+1}=k^{n+1}-k^{n}$, $\overline{\varepsilon}^{n+1}=\varepsilon^{n+1}-\varepsilon^{n}$,\
$\overline{p}^{n+1}=p^{n+1}-p^{n}=(\rho^{n+1})^{\gamma}-(\rho^{n})^{\gamma}$.

Then, by equations (\ref{2.1})-(\ref{2.5}), we deduce that ($\overline{\rho}^{n+1}$,\ $\overline{u}^{n+1}$,
$\overline{h}^{n+1}$,\ $\overline{k}^{n+1}$,\ $\overline{\varepsilon}^{n+1}$,\ $\overline{p}^{n+1}$) satisfy the following equations:
\begin{eqnarray}
&&\overline{\rho}^{n+1}_{t}+\nabla\cdot(\overline{\rho}^{n+1}u^{n}+\rho^{n}\overline{u}^{n})=0,\label{dif1}\\
&&\rho^{n+1}\overline{u}^{n+1}_{t}+\overline{\rho}^{n+1}u^{n}_{t}
+\rho^{n+1}u^{n}\cdot\nabla\overline{u}^{n+1}
+\overline{\rho}^{n+1}u^{n}\cdot\nabla u^{n}
+\rho^{n}\overline{u}^{n}\cdot\nabla u^{n}\nonumber\\
&&-\Delta \overline{u}^{n+1}
-\nabla(\nabla\cdot\overline{u}^{n+1})
+\nabla\overline{p}^{n+1}
=\frac{-2}{3}\nabla(\overline{\rho}^{n+1}k^{n}+\rho^{n}\overline{k}^{n}),\label{dif2}\\
&&\rho^{n+1}\overline{h}^{n+1}_{t}+\overline{\rho}^{n+1}h^{n}_{t}
+\rho^{n+1}u^{n}\cdot\nabla\overline{h}^{n+1}
+\overline{\rho}^{n+1}u^{n}\cdot\nabla h^{n}
+\rho^{n}\overline{u}^{n}\cdot\nabla h^{n}\nonumber\\
&&-\Delta \overline{h}^{n+1}
=\overline{p}^{n+1}_{t}+\overline{u}^{n+1}\cdot\nabla p^{n+1}+u^{n}\cdot\nabla\overline{p}^{n+1}
+S_{k,n+1}^{'}-S_{k,n}^{'},\label{dif3}
\end{eqnarray}
\begin{eqnarray}
&&\rho^{n+1}\overline{k}^{n+1}_{t}+\overline{\rho}^{n+1}k^{n}_{t}
+\rho^{n+1}u^{n}\cdot\nabla\overline{k}^{n+1}
+\overline{\rho}^{n+1}u^{n}\cdot\nabla k^{n}
+\rho^{n}\overline{u}^{n}\cdot\nabla k^{n}\nonumber\\
&&-\Delta \overline{k}^{n+1}
=G_{n+1}^{'}-G_{n}^{'}-(\rho^{n+1}\varepsilon^{n}-\rho^{n}\varepsilon^{n-1}),\label{dif4}\\
&&\rho^{n+1}\overline{\varepsilon}^{n+1}_{t}+\overline{\rho}^{n+1}\varepsilon^{n}_{t}
+\rho^{n+1}u^{n}\cdot\nabla\overline{\varepsilon}^{n+1}
+\overline{\rho}^{n+1}u^{n}\cdot\nabla \varepsilon^{n}
+\rho^{n}\overline{u}^{n}\cdot\nabla \varepsilon^{n}\nonumber\\
&&-\Delta \overline{\varepsilon}^{n+1}
=C_{1}\bigg(\frac{G_{n+1}^{'} \varepsilon^{n}}{k^{n}}
-\frac{G_{n}^{'} \varepsilon^{n-1}}{k^{n-1}}\bigg)
-C_{2}\bigg(\frac{\rho^{n+1}(\varepsilon^{n})^{2}}{k^{n}}
-\frac{\rho^{n}(\varepsilon^{n-1})^{2}}{k^{n-1}}\bigg),\label{dif5}
\end{eqnarray}
where \be
S_{k,n+1}^{'}=[\mu(\partial_{j}u_{i}^{n}+\partial_{i}u_{j}^{n})
-\frac{2}{3}\delta_{ij}\mu\partial_{k}u_{k}^{n}]\partial_{j}u_{i}^{n}
+\frac{\mu_{t}}{(\rho^{n+1})^{2}}\partial_{j}p^{n+1}\partial_{j}\rho^{n+1},
\ee
\be
G_{n+1}^{'}=\partial_{j}u_{i}^{n}[\mu_{e}(\partial_{j}u_{i}^{n}+\partial_{i}u_{j}^{n})
-\frac{2}{3}\delta_{ij}(\rho^{n+1}k^{n}+\mu_{e}\partial_{l}u_{l}^{n})].
\ee

To evaluate $\|\overline{\rho}^{n+1}\|_{L^{2}}$, multiplying both
sides of equation (\ref{dif1}) by $\overline{\rho}^{n+1}$ and
integrating  the result  over $\Omega$,  we get
\begin{eqnarray}
&&\frac{1}{2}\frac{\mbox{d}}{\mbox{d}t}\|\overline{\rho}^{n+1}\|_{L^{2}}^{2}
=-\int\nabla\cdot(\overline{\rho}^{n+1}u^{n}+\rho^{n}\overline{u}^{n})\cdot\overline{\rho}^{n+1}\nonumber\\
&&=-\int (\overline{\rho}^{n+1})^{2}\nabla\cdot u^{n}
+\overline{\rho}^{n+1}u^{n}\cdot\nabla\overline{\rho}^{n+1}
+\rho^{n}\overline{\rho}^{n+1}\nabla\cdot \overline{u}^{n}
+\overline{\rho}^{n+1}\overline{u}^{n}\cdot\nabla\rho^{n}.\label{rounplus1}
\end{eqnarray}

Applying integration by parts to the second term of the second equality of (\ref{rounplus1}) and
using H\"{o}lder, Sobolev and Young's inequalities yield
\begin{eqnarray}
\frac{\mbox{d}}{\mbox{d}t}\|\overline{\rho}^{n+1}\|_{L^{2}}^{2}
&\leq& C(\|\nabla u^{n}\|_{L^{\infty}}\|\overline{\rho}^{n+1}\|_{L^{2}}^{2}
+\|\nabla \overline{u}^{n}\|_{L^{2}}\|\overline{\rho}^{n+1}\|_{L^{2}}
+\|\overline{u}^{n}\|_{L^{6}}\|\nabla\rho^{n}\|_{L^{3}}\|\overline{\rho}^{n+1}\|_{L^{2}})\nonumber\\
&\leq& \widetilde{C}(1+\eta^{-1})\|\overline{\rho}^{n+1}\|_{L^{2}}^{2}
+\widetilde{C}\eta\|\nabla \overline{u}^{n}\|_{H^{1}}^{2},\label{drou}
\end{eqnarray}
where (\ref{another conclusion}) has been used and $0<\eta<1$  is a small constant to be determined later.

Next, multiplying both sides of (\ref{dif2}) by $\overline{u}^{n+1}$ and integrating the result thus derived over $\Omega$,  one obtains
\begin{eqnarray}
&&\frac{1}{2}\frac{\mbox{d}}{\mbox{d}t}\|\sqrt{\rho^{n+1}}\overline{u}^{n+1}\|_{L^{2}}^{2}
+\|\nabla\overline{u}^{n+1}\|_{L^{2}}^{2}
+\|\nabla\cdot\overline{u}^{n+1}\|_{L^{2}}^{2}\nonumber\\
&&=-\int \overline{\rho}^{n+1}u^{n}_{t}\cdot\overline{u}^{n+1}
-\int\overline{\rho}^{n+1}u^{n}\cdot\nabla u^{n}\cdot\overline{u}^{n+1}
-\int\rho^{n}\overline{u}^{n}\cdot\nabla u^{n}\cdot\overline{u}^{n+1}
-\int\nabla\overline{p}^{n+1}\cdot\overline{u}^{n+1}\nonumber\\
&&+\int\frac{-2}{3}\nabla(\overline{\rho}^{n+1}k^{n}+
\rho^{n}\overline{k}^{n+1})\cdot\overline{u}^{n+1}
=\mathop{\sum}_{i=1}^{5}L_{i}.\label{difu}
\end{eqnarray}

Using H\"{o}lder, Sobolev and Young's inequalities and (\ref{another conclusion}), we estimate  $L_{1}$, $L_{2}$ and $L_{3}$, respectively, as follows:
\begin{eqnarray}
&&L_{1}\leq C\|\overline{\rho}^{n+1}\|_{L^{2}}\|u_{t}^{n}\|_{L^{3}}\|\overline{u}^{n+1}\|_{L^{6}}
\leq C\|\overline{\rho}^{n+1}\|_{L^{2}}\|u_{t}^{n}\|_{L^{3}}(\|\sqrt{\rho^{n+1}}\overline{u}^{n+1}\|_{L^{2}}+\|\nabla\overline{u}^{n+1}\|_{L^{2}})\nonumber\\
&&\leq \widetilde{C}\|u_{t}^{n}\|_{L^{3}}^{2}\|\overline{\rho}^{n+1}\|_{L^{2}}^{2}
+\widetilde{C}\|\sqrt{\rho^{n+1}}\overline{u}^{n+1}\|_{L^{2}}^{2}
+\frac{1}{8}\|\nabla\overline{u}^{n+1}\|_{L^{2}}^{2},\label{L1}
\end{eqnarray}
\begin{eqnarray}
&&L_{2}\leq C\|\overline{\rho}^{n+1}\|_{L^{2}}\|u^{n}\|_{L^{6}}\|\nabla u^{n}\|_{L^{6}}\|\overline{u}^{n+1}\|_{L^{6}}\nonumber\\
&&\leq \widetilde{C}\|\overline{\rho}^{n+1}\|_{L^{2}}^{2}
+\widetilde{C}\|\sqrt{\rho^{n+1}}\overline{u}^{n+1}\|_{L^{2}}^{2}+\frac{1}{8}\|\nabla\overline{u}^{n+1}\|_{L^{2}}^{2},\label{L2}
\end{eqnarray}
\begin{eqnarray}
L_{3}
\leq C\|\overline{u}^{n}\|_{L^{6}}
\|\nabla u^{n}\|_{L^{3}}\|\sqrt{\rho^{n+1}}\overline{u}^{n+1}\|_{L^{2}}
\leq \widetilde{C}\eta^{-1}\|\sqrt{\rho^{n+1}}\overline{u}^{n+1}\|_{L^{2}}^{2}
+\eta\|\overline{u}^{n}\|_{H^{1}}^{2}.\label{L3}
\end{eqnarray}
And then, one deduces by integration by parts that
\begin{eqnarray}
L_{4}=\int \overline{p}^{n+1}\nabla\cdot\overline{u}^{n+1}
\leq C\int\overline{\rho}^{n+1}\nabla\cdot\overline{u}^{n+1}
\leq \widetilde{C}\|\overline{\rho}^{n+1}\|_{L^{2}}^{2}+\frac{1}{8}\|\nabla\overline{u}^{n+1}\|_{L^{2}}^{2},\label{L4}
\end{eqnarray}
and
\begin{eqnarray}\nonumber
&&L_{5}=\frac{2}{3}\int\overline{\rho}^{n+1}k^{n}\nabla\cdot\overline{u}^{n+1}
-\overline{k}^{n}\nabla\rho^{n}\cdot\overline{u}^{n+1}
-\rho^{n}\nabla\overline{k}^{n}\cdot\overline{u}^{n+1}\\
&&\leq C\|\overline{\rho}^{n+1}\|_{L^{2}}\|\nabla\overline{u}^{n+1}\|_{L^{2}}
+C\|\overline{k}^{n}\|_{L^{6}}\|\nabla\rho^{n}\|_{L^{3}}\|\sqrt{\rho^{n+1}}\overline{u}^{n+1}\|_{L^{2}}
+C\|\nabla\overline{k}^{n}\|_{L^{2}}\|\sqrt{\rho^{n+1}}\overline{u}^{n+1}\|_{L^{2}}\nonumber\\
&&\leq \widetilde{C}(1+\eta^{-1})(\|\overline{\rho}^{n+1}\|_{L^{2}}^{2}
+\|\sqrt{\rho^{n+1}}\overline{u}^{n+1}\|_{L^{2}}^{2})
+\frac{1}{8}\|\nabla\overline{u}^{n+1}\|_{L^{2}}^{2}
+\widetilde{C}\eta\|\overline{k}^{n}\|_{H^{1}}^{2}.\label{L5}
\end{eqnarray}
Inserting (\ref{L1})-(\ref{L5}) to (\ref{difu}) and using inequality
$\|\overline{u}^{n+1}\|_{L^{2}}\leq \widetilde{C}\|\sqrt{\rho^{n+1}}\overline{u}^{n+1}\|_{L^{2}}$, one has
\begin{eqnarray}
&&\frac{\mbox{d}}{\mbox{d}t}\|\sqrt{\rho^{n+1}}\overline{u}^{n+1}\|_{L^{2}}^{2}
+\|\overline{u}^{n+1}\|_{H^{1}}^{2}\label{difu-2}\\
&&\leq \widetilde{C}(1+\eta^{-1}+\|u_{t}^{n}\|_{L^{3}}^{2})(\|\overline{\rho}^{n+1}\|_{L^{2}}^{2}
+\|\sqrt{\rho^{n+1}}\overline{u}^{n+1}\|_{L^{2}}^{2})
+\widetilde{C}\eta\|\overline{k}^{n}\|_{H^{1}}^{2}
+\widetilde{C}\eta\|\overline{u}^{n}\|_{H^{1}}^{2}.\nonumber
\end{eqnarray}

Then, multiplying both sides of (\ref{dif3}) by $\overline{h}^{n+1}$ and integrating the result thus got over $\Omega$,  one obtains
\begin{eqnarray}
&&\frac{1}{2}\frac{\mbox{d}}{\mbox{d}t}\|\sqrt{\rho^{n+1}}\overline{h}^{n+1}\|_{L^{2}}^{2}
+\|\nabla\overline{h}^{n+1}\|_{L^{2}}^{2}\nonumber\\
&&=-\int \overline{\rho}^{n+1}h^{n}_{t}\cdot\overline{h}^{n+1}
-\int\overline{\rho}^{n+1}u^{n}\cdot\nabla h^{n}\cdot\overline{h}^{n+1}
-\int\rho^{n}\overline{u}^{n}\cdot\nabla h^{n}\cdot\overline{h}^{n+1}\label{difh}\\
&&+\int (\overline{p}^{n+1}_{t}+\overline{u}^{n+1}\cdot\nabla p^{n+1}+u^{n}\cdot\nabla\overline{p}^{n+1})\cdot\overline{h}^{n+1}
+\int (S_{k,n+1}^{'}-S_{k,n}^{'})\cdot\overline{h}^{n+1}
=\mathop{\sum}_{i=1}^{5}M_{i}.\nonumber
\end{eqnarray}
First, using similar method of deriving (\ref{L1}), (\ref{L2}) and (\ref{L3}), respectively, one easily obtains
\begin{eqnarray}
M_{1}
\leq \widetilde{C}\|h_{t}^{n}\|_{L^{3}}^{2}\|\overline{\rho}^{n+1}\|_{L^{2}}^{2}
+\widetilde{C}\|\sqrt{\rho^{n+1}}\overline{h}^{n+1}\|_{L^{2}}^{2}+\frac{1}{20}\|\nabla\overline{h}^{n+1}\|_{L^{2}}^{2},\label{M1}
\end{eqnarray}
\begin{eqnarray}
M_{2}
\leq \widetilde{C}\|\overline{\rho}^{n+1}\|_{L^{2}}^{2}
+\widetilde{C}\|\sqrt{\rho^{n+1}}\overline{h}^{n+1}\|_{L^{2}}^{2}+\frac{1}{20}\|\nabla\overline{h}^{n+1}\|_{L^{2}}^{2},\label{M2}
\end{eqnarray}
\begin{eqnarray}
M_{3}\leq\widetilde{C}\eta^{-1}\|\sqrt{\rho^{n+1}}\overline{h}^{n+1}\|_{L^{2}}^{2}
+\eta\|\overline{u}^{n}\|_{H^{1}}^{2}.\label{M3}
\end{eqnarray}
Second, simple calculation leads to
\begin{eqnarray}
&&M_{4}=\int [\gamma(\rho^{n+1})^{\gamma-1}\rho^{n+1}_{t}-\gamma(\rho^{n})^{\gamma-1}\rho^{n}_{t}]\cdot\overline{h}^{n+1}
+\int \overline{u}^{n+1}\cdot\nabla p^{n+1}\overline{h}^{n+1}\nonumber\\
&&+\int u^{n}\cdot\nabla \overline{p}^{n+1}\overline{h}^{n+1}.\label{MM4-1}
\end{eqnarray}
By the differential mean value theorem, the first integral of (\ref{MM4-1}) can be controlled as
\begin{eqnarray}
&&\int [\gamma(\rho^{n+1})^{\gamma-1}\rho^{n+1}_{t}-\gamma(\rho^{n})^{\gamma-1}\rho^{n}_{t}]\cdot\overline{h}^{n+1}\nonumber\\
&&\leq C\int|\overline{\rho}^{n+1}||\rho^{n+1}_{t}||\overline{h}^{n+1}|
+\int\gamma(\rho^{n})^{\gamma-1}\overline{\rho}^{n+1}_{t}\cdot\overline{h}^{n+1}.\label{MM4-2}
\end{eqnarray}
By equation (\ref{dif1}), the second integral on the right hand side of (\ref{MM4-2}) can be estimated as
\begin{eqnarray}
&&\int\gamma(\rho^{n})^{\gamma-1}\overline{\rho}^{n+1}_{t}\cdot\overline{h}^{n+1}
=-\int\gamma(\rho^{n})^{\gamma-1}\nabla\cdot(\overline{\rho}^{n+1}u^{n}+\rho^{n}\overline{u}^{n})\cdot\overline{h}^{n+1}\nonumber\\
&&\leq C\int|\nabla\rho^{n}||\overline{h}^{n+1}||\overline{\rho}^{n+1}||u^{n}|
+C\int|\overline{\rho}^{n+1}||u^{n}||\nabla\overline{h}^{n+1}|\nonumber\\
&&+C\int(|\nabla\rho^{n}||\overline{u}^{n}|+|\rho^{n}||\nabla\overline{u}^{n}|)|\overline{h}^{n+1}|.\label{MM4-3}
\end{eqnarray}
Then, the second integral on the right hand side of  (\ref{MM4-1}) can be controlled as
\begin{eqnarray}
\int \overline{u}^{n+1}\cdot\nabla p^{n+1}\overline{h}^{n+1}
\leq C\int|\overline{u}^{n+1}||\nabla \rho^{n+1}||\overline{h}^{n+1}|.
\end{eqnarray}

Next, applying integration by parts to the third integral on the right hand side of (\ref{MM4-1}), we easily get
\begin{eqnarray}
\int u^{n}\cdot\nabla \overline{p}^{n+1}\overline{h}^{n+1}
\leq C\int|\nabla u^{n}||\overline{\rho}^{n+1}||\overline{h}^{n+1}|
+C\int|u^{n}||\overline{\rho}^{n+1}||\nabla \overline{h}^{n+1}|.\label{MM4-4}
\end{eqnarray}
Consequently, combining (\ref{MM4-1})-(\ref{MM4-4}) and using H\"{o}lder, Sobolev and Young's inequalities and
 (\ref{another conclusion}), one obtains
\begin{eqnarray}
&&M_{4}\leq \widetilde{C}(1+\eta^{-1})(\|\overline{\rho}^{n+1}\|_{L^{2}}^{2}
+\|\sqrt{\rho^{n+1}}\overline{h}^{n+1}\|_{L^{2}}^{2}
)\nonumber\\
&&+\frac{1}{4}\|\overline{u}^{n+1}\|_{H^{1}}^{2}
+\frac{1}{20}\|\nabla\overline{h}^{n+1}\|_{L^{2}}^{2}+\widetilde{C}\eta\|\overline{u}^{n}\|_{H^{1}}^{2}.\label{M4-2}
\end{eqnarray}

Finally, we  evaluate $M_{5}$.
Direct calculation yields
\begin{eqnarray}
&&M_{5}
\leq C\int(|\nabla u^{n}|+|\nabla u^{n-1}|)|\nabla \overline{u}^{n}||\overline{h}^{n+1}|
+C\int|\overline{\rho}^{n+1}||\nabla\rho^{n+1}|^{2}|\overline{h}^{n+1}|\nonumber\\
&&+\int\frac{\mu_{t}}{(\rho^{n})^{2}}\partial_{j}\overline{p}^{n+1}\partial_{j}\rho^{n+1}\cdot\overline{h}^{n+1}
+\int\frac{\mu_{t}}{(\rho^{n})^{2}}\partial_{j}p^{n}\partial_{j}\overline{\rho}^{n+1}\cdot\overline{h}^{n+1}\nonumber\\
&&\leq C\int(|\nabla u^{n}|+|\nabla u^{n-1}|)|\nabla \overline{u}^{n}||\overline{h}^{n+1}|
+C\int|\overline{\rho}^{n+1}||\nabla\rho^{n+1}|^{2}|\overline{h}^{n+1}|\nonumber\\
&&+C\int|\nabla\rho^{n}||\nabla\rho^{n+1}||\overline{\rho}^{n+1}||\overline{h}^{n+1}|
+C\int|\nabla^{2}\rho^{n+1}||\overline{\rho}^{n+1}||\overline{h}^{n+1}|\nonumber\\
&&+C\int|\nabla\rho^{n+1}||\overline{\rho}^{n+1}||\nabla\overline{h}^{n+1}|
+C\int|\nabla\rho^{n}|^{2}||\overline{\rho}^{n+1}||\overline{h}^{n+1}|\nonumber\\
&&+C\int|\nabla^{2}\rho^{n}||\overline{\rho}^{n+1}||\overline{h}^{n+1}|
+C\int|\nabla\rho^{n}||\overline{\rho}^{n+1}||\nabla\overline{h}^{n+1}|.\label{M5}
\end{eqnarray}

Then, applying similar method of deriving (\ref{M4-2}), one deduces
\begin{eqnarray}
M_{5}\leq \widetilde{C}(1+\eta^{-1})(\|\overline{\rho}^{n+1}\|_{L^{2}}^{2}
+\|\sqrt{\rho^{n+1}}\overline{h}^{n+1}\|_{L^{2}}^{2})
+\eta\|\overline{u}^{n}\|_{H^{1}}^{2}+\frac{1}{20}\|\nabla\overline{h}^{n+1}\|_{L^{2}}^{2}.\label{M5-2}
\end{eqnarray}

Consequently, inserting (\ref{M1})-(\ref{M3}), (\ref{M4-2}) and (\ref{M5-2}) to (\ref{difh}), one gets
\begin{eqnarray}
&&\frac{\mbox{d}}{\mbox{d}t}\|\sqrt{\rho^{n+1}}\overline{h}^{n+1}\|_{L^{2}}^{2}
+\|\overline{h}^{n+1}\|_{H^{1}}^{2}\nonumber\\
&&\leq \widetilde{C}(1+\eta^{-1}+\|h_{t}^{n}\|_{L^{3}}^{2})(\|\overline{\rho}^{n+1}\|_{L^{2}}^{2}
+\|\sqrt{\rho^{n+1}}\overline{h}^{n+1}\|_{L^{2}}^{2})\nonumber\\
&&+\frac{1}{4}\|\overline{u}^{n+1}\|_{H^{1}}^{2}
+\tilde{C}\eta\|\overline{u}^{n}\|_{H^{1}}^{2}.\label{difh-2}
\end{eqnarray}

For the turbulent kinetic energy $k$, using similar method of
deriving (\ref{difh}), one easily deduces from equation (\ref{dif4})
that
\begin{eqnarray}
&&\frac{1}{2}\frac{\mbox{d}}{\mbox{d}t}\|\sqrt{\rho^{n+1}}\overline{k}^{n+1}\|_{L^{2}}^{2}
+\|\nabla\overline{k}^{n+1}\|_{L^{2}}^{2}
=-\int \overline{\rho}^{n+1}k^{n}_{t}\cdot\overline{k}^{n+1}
-\int\overline{\rho}^{n+1}u^{n}\cdot\nabla k^{n}\cdot\overline{k}^{n+1}\label{difk}\\
&&-\int\rho^{n}\overline{u}^{n}\cdot\nabla k^{n}\cdot\overline{k}^{n+1}
+\int(G_{n+1}^{'}-G_{n}^{'})\cdot\overline{k}^{n+1}
-\int(\rho^{n+1}\varepsilon^{n}-\rho^{n}\varepsilon^{n-1})\cdot\overline{k}^{n+1}
=\mathop{\sum}_{i=1}^{5}N_{i}.\nonumber
\end{eqnarray}

We first evaluate $N_{4}$. Using inserting items technic, one easily
gets
\begin{eqnarray}
&&N_{4}\leq C\int(|\nabla u^{n}|+|\nabla u^{n-1}|)|\nabla \overline{u}^{n}||\overline{k}^{n+1}|\nonumber\\
&&+C\int(|\nabla \overline{u}^{n}|+|\nabla u^{n-1}||\overline{\rho}^{n+1}|
+|\nabla u^{n-1}||\overline{k}^{n}|)|\overline{k}^{n+1}|.\label{NN4-1}
\end{eqnarray}

Using H\"{o}lder, Sobolev, and Young's inequalities and
(\ref{another conclusion}), we have
\begin{eqnarray}
N_{4}\leq \widetilde{C}(1+\eta^{-1})(\|\overline{\rho}^{n+1}\|_{L^{2}}^{2}
+\|\sqrt{\rho^{n+1}}\overline{k}^{n+1}\|_{L^{2}}^{2})
+\widetilde{C}\eta\|\overline{k}^{n}\|_{H^{1}}^{2}
+\widetilde{C}\eta\|\overline{u}^{n}\|_{H^{1}}^{2}.\label{NN4-2}
\end{eqnarray}

 Second, we estimate $N_{5}$. Using similar method of deriving (\ref{NN4-1}) and (\ref{NN4-2}), we have
 \begin{eqnarray}
 &&N_{5}=\int(\overline{\rho}^{n+1}\varepsilon^{n}
 +\rho^{n}\overline{\varepsilon}^{n})\cdot\overline{k}^{n+1}
 \leq C(\|\overline{\rho}^{n+1}\|_{L^{2}}\|\varepsilon^{n}\|_{L^{\infty}}
 +\|\overline{\varepsilon}^{n}\|_{L^{6}}\|\rho^{n}\|_{L^{3}})\|\sqrt{\rho^{n+1}}\overline{k}^{n+1}\|_{L^{2}}\nonumber\\
 &&\leq \widetilde{C}(1+\eta^{-1})(\|\sqrt{\rho^{n+1}}\overline{k}^{n+1}\|_{L^{2}}^{2}+\|\overline{\rho}^{n+1}\|_{L^{2}}^{2})
+\widetilde{C}\eta\|\overline{\varepsilon}^{n}\|_{H^{1}}^{2}.\label{NN5}
 \end{eqnarray}

Next, using similar method of deriving  the estimates of (\ref{L1}), (\ref{L2}) and (\ref{L3}), respectively, one easily gets
\begin{eqnarray}
N_{1}
\leq \widetilde{C}\|k_{t}^{n}\|_{L^{3}}^{2}\|\overline{\rho}^{n+1}\|_{L^{2}}^{2}
+\widetilde{C}\|\sqrt{\rho^{n+1}}\overline{k}^{n+1}\|_{L^{2}}
+\frac{1}{8}\|\nabla\overline{k}^{n+1}\|_{L^{2}}^{2},\label{N1}
\end{eqnarray}
\begin{eqnarray}
N_{2}
\leq \widetilde{C}\|\overline{\rho}^{n+1}\|_{L^{2}}^{2}
+\widetilde{C}\|\sqrt{\rho^{n+1}}\overline{k}^{n+1}\|_{L^{2}}^{2}
+\frac{1}{8}\|\nabla\overline{k}^{n+1}\|_{L^{2}}^{2},\label{N2}
\end{eqnarray}
\begin{eqnarray}
N_{3}\leq\widetilde{C}\eta^{-1}\|\sqrt{\rho^{n+1}}\overline{k}^{n+1}\|_{L^{2}}^{2}
+\eta\|\overline{u}^{n}\|_{H^{1}}^{2}.\label{N3}
\end{eqnarray}

Consequently, inserting (\ref{NN4-2})-(\ref{N3}) to (\ref{difk}),
one deduces
\begin{eqnarray}
&&\frac{\mbox{d}}{\mbox{d}t}\|\sqrt{\rho^{n+1}}\overline{k}^{n+1}\|_{L^{2}}^{2}
+\|\overline{k}^{n+1}\|_{H^{1}}^{2}\label{difk-2}\\
&&\leq \widetilde{C}(1+\eta^{-1}+\|k_{t}^{n}\|_{L^{3}}^{2})(\|\sqrt{\rho^{n+1}}\overline{k}^{n+1}\|_{L^{2}}^{2}+\|\overline{\rho}^{n+1}\|_{L^{2}}^{2})
+\widetilde{C}\eta(\|\overline{k}^{n}\|_{H^{1}}^{2}
+\|\overline{u}^{n}\|_{H^{1}}^{2}
+\|\overline{\varepsilon}^{n}\|_{H^{1}}^{2}).\nonumber
\end{eqnarray}

Next, multiplying both sides of (\ref{dif5}) by  $\overline{\varepsilon}^{n+1}$  and integrating  the result
 over $\Omega$, one gets

\begin{eqnarray}
&&\frac{1}{2}\frac{\mbox{d}}{\mbox{d}t}\|\sqrt{\rho^{n+1}}\overline{\varepsilon}^{n+1}\|_{L^{2}}^{2}
+\|\nabla\overline{\varepsilon}^{n+1}\|_{L^{2}}^{2}
=-\int \overline{\rho}^{n+1}\varepsilon^{n}_{t}\cdot\overline{\varepsilon}^{n+1}
-\int\overline{\rho}^{n+1}u^{n}\cdot\nabla \varepsilon^{n}\cdot\overline{\varepsilon}^{n+1}\nonumber\\
&&-\int\rho^{n}\overline{u}^{n}\cdot\nabla \varepsilon^{n}\cdot\overline{\varepsilon}^{n+1}
+C_{1}\int\bigg(\frac{G_{n+1}^{'}\varepsilon^{n}}{k^{n}}-\frac{G_{n}^{'}\varepsilon^{n-1}}{k^{n-1}}\bigg)\cdot\overline{\varepsilon}^{n+1}\nonumber\\
&&-C_{2}\int \bigg[\frac{\rho^{n+1}(\varepsilon^{n})^{2}}{k^{n}}
-\frac{\rho^{n}(\varepsilon^{n-1})^{2}}{k^{n-1}}\bigg]\cdot\overline{\varepsilon}^{n+1}
=\mathop{\sum}_{i=1}^{5}Q_{i}.\label{dife}
\end{eqnarray}

Using  an argument similar to that used in deriving (\ref{L1}),
(\ref{L2})  and (\ref{L3}), respectively, we obtain
 \begin{eqnarray}
Q_{1}
\leq \widetilde{C}\|\varepsilon_{t}^{n}\|_{L^{3}}^{2}\|\overline{\rho}^{n+1}\|_{L^{2}}^{2}
+\widetilde{C}\|\sqrt{\rho^{n+1}}\overline{\varepsilon}^{n+1}\|_{L^{2}}^{2}
+\frac{1}{8}\|\nabla\overline{\varepsilon}^{n+1}\|_{L^{2}}^{2},\label{Q1}
\end{eqnarray}
\begin{eqnarray}
Q_{2}
\leq \widetilde{C}\|\overline{\rho}^{n+1}\|_{L^{2}}^{2}
+\widetilde{C}\|\sqrt{\rho^{n+1}}\overline{\varepsilon}^{n+1}\|_{L^{2}}^{2}
+\frac{1}{8}\|\nabla\overline{\varepsilon}^{n+1}\|_{L^{2}}^{2},\label{Q2}
\end{eqnarray}
\begin{eqnarray}
Q_{3}
\leq \widetilde{C}\eta^{-1}\|\sqrt{\rho^{n+1}}\overline{\varepsilon}^{n+1}\|_{L^{2}}^{2}
+\widetilde{C}\eta\|\overline{u}^{n}\|_{H^{1}}^{2}.\label{Q3}
\end{eqnarray}

Next, direct calculation leads to
\begin{eqnarray}
&&Q_{4}
\leq C\int(|\nabla \overline{u}^{n}||\nabla u^{n}|+
|\nabla \overline{u}^{n}||\nabla u^{n-1}|
)|\varepsilon^{n}||\overline{\varepsilon}^{n+1}|
+C\int(|\overline{\varepsilon}^{n}|+|\varepsilon^{n-1}||\overline{k}^{n}|
)|\nabla u^{n-1}|^{2}|\overline{\varepsilon}^{n+1}|\nonumber\\
&&-\frac{2C_{1}}{3}\delta_{ij}\int\frac{(\partial_{j}u_{i}^{n}\rho^{n+1}k^{n}\varepsilon^{n}k^{n-1}
-\partial_{j}u_{i}^{n-1}\rho^{n}k^{n-1}\varepsilon^{n-1}k^{n})}{k^{n}k^{n-1}}\cdot\overline{\varepsilon}^{n+1}\nonumber\\
&&\leq \int(|\nabla \overline{u}^{n}||\nabla u^{n}|+
|\nabla \overline{u}^{n}||\nabla u^{n-1}|
)|\varepsilon^{n}||\overline{\varepsilon}^{n+1}|\nonumber\\
&&+C\int(|\overline{\varepsilon}^{n}|+|\varepsilon^{n-1}||\overline{k}^{n}|
)|\nabla u^{n-1}|^{2}|\overline{\varepsilon}^{n+1}|\nonumber\\
&&+C\int(|\nabla \overline{u}^{n}|
+|\nabla u^{n-1}||\overline{\rho}^{n+1}|
+|\nabla u^{n-1}||\overline{k}^{n}|)|\varepsilon^{n}||\overline{\varepsilon}^{n+1}|\nonumber\\
&&+C\int(|\overline{\varepsilon}^{n}|
+|\varepsilon^{n-1}||\overline{k}^{n}|
)|\nabla u^{n-1}||\overline{\varepsilon}^{n+1}|\label{Q4-2}\\
&&\leq \widetilde{C}(1+\eta^{-1})(\|\sqrt{\rho^{n+1}}\overline{\varepsilon}^{n+1}\|_{L^{2}}^{2}+\|\overline{\rho}^{n+1}\|_{L^{2}}^{2})\nonumber\\
&&+\widetilde{C}\eta(\|\overline{u}^{n}\|_{H^{1}}^{2}
+\|\overline{k}^{n}\|_{H^{1}}^{2}
+\|\overline{\varepsilon}^{n}\|_{H^{1}}^{2})
+\frac{1}{8}\|\nabla\overline{\varepsilon}^{n+1}\|_{L^{2}}^{2}.\nonumber
\end{eqnarray}

Finally, using similar method in deriving the estimate of $Q_{4}$,
one deduces
\begin{eqnarray}
&&Q_{5}
\leq \widetilde{C}(1+\eta^{-1})(\|\sqrt{\rho^{n+1}}\overline{\varepsilon}^{n+1}\|_{L^{2}}^{2}+\|\overline{\rho}^{n+1}\|_{L^{2}}^{2})
+\widetilde{C}\eta\|\nabla\overline{\varepsilon}^{n}\|_{L^{2}}^{2}
+\frac{1}{8}\|\nabla\overline{\varepsilon}^{n+1}\|_{L^{2}}^{2}.\label{Q5}
\end{eqnarray}

Consequently, inserting (\ref{Q1})-(\ref{Q5}) to (\ref{dife}), one
derives
\begin{eqnarray}
&&\frac{\mbox{d}}{\mbox{d}t}\|\sqrt{\rho^{n+1}}\overline{\varepsilon}^{n+1}\|_{L^{2}}^{2}
+\|\overline{\varepsilon}^{n+1}\|_{H^{1}}^{2}\label{dife-2}\\
&&\leq \widetilde{C}(1+\eta^{-1}+\|\varepsilon_{t}^{n}\|_{L^{3}}^{2})
(\|\sqrt{\rho^{n+1}}\overline{\varepsilon}^{n+1}\|_{L^{2}}^{2}+\|\overline{\rho}^{n+1}\|_{L^{2}}^{2})
+\widetilde{C}\eta(\|\overline{k}^{n}\|_{H^{1}}^{2}
+\|\overline{u}^{n}\|_{H^{1}}^{2}
+\|\overline{\varepsilon}^{n}\|_{H^{1}}^{2}).\nonumber
\end{eqnarray}

In the end, combining (\ref{drou}), (\ref{difu-2}), (\ref{difh-2}), (\ref{difk-2}) and (\ref{dife-2}) and setting
$
\varphi^{n+1}(t)=\|\overline{\rho}^{n+1}\|_{L^{2}}^{2}
+\|\sqrt{\rho^{n+1}}\overline{u}^{n+1}\|_{L^{2}}^{2}
+\|\sqrt{\rho^{n+1}}\overline{h}^{n+1}\|_{L^{2}}^{2}
+\|\sqrt{\rho^{n+1}}\overline{k}^{n+1}\|_{L^{2}}^{2}
+\|\sqrt{\rho^{n+1}}\overline{\varepsilon}^{n+1}\|_{L^{2}}^{2}$,\,
we get
\begin{eqnarray}
&&\frac{\mbox{d}}{\mbox{d}t}\varphi^{n+1}(t)
+\|\overline{u}^{n+1}\|_{H^{1}}^{2}
+\|\overline{h}^{n+1}\|_{H^{1}}^{2}
+\|\overline{k}^{n+1}\|_{H^{1}}^{2}
+\|\overline{\varepsilon}^{n+1}\|_{H^{1}}^{2}\label{con2}\\
&&\leq \widetilde{C}(1+\eta^{-1}+\|u_{t}^{n}\|_{L^{3}}^{2}
+\|h_{t}^{n}\|_{L^{3}}^{2}
+\|k_{t}^{n}\|_{L^{3}}^{2}
+\|\varepsilon_{t}^{n}\|_{L^{3}}^{2})\varphi^{n+1}(t)\nonumber\\
&&+\widetilde{C}\eta(\|\overline{u}^{n}\|_{H^{1}}^{2}
+\|\overline{k}^{n}\|_{H^{1}}^{2}
+\|\overline{\varepsilon}^{n}\|_{H^{1}}^{2}).\nonumber
\end{eqnarray}

Setting $I_{\eta}^{n}(t)=\widetilde{C}(1+\eta^{-1}+\|u_{t}^{n}\|_{L^{3}}^{2}
+\|h_{t}^{n}\|_{L^{3}}^{2}
+\|k_{t}^{n}\|_{L^{3}}^{2}
+\|\varepsilon_{t}^{n}\|_{L^{3}}^{2})$ and applying Gronwall's inequality to (\ref{con2}) yield
\be
\varphi^{n+1}(t)
\leq \widetilde{C}\eta\bigg[\exp\bigg(\int_{0}^{t}I_{\eta}^{n}(s)\mbox{d}s\bigg)\bigg]\bigg(\int_{0}^{t}(\|\overline{u}^{n}\|_{H^{1}}^{2}
+\|\overline{k}^{n}\|_{H^{1}}^{2}
+\|\overline{\varepsilon}^{n}\|_{H^{1}}^{2})\mbox{d}s\bigg),\label{con3}
\ee
where it should be noted that $\varphi^{n+1}(0)=0$.

Since
\be
\int_{0}^{t}I_{\eta}^{n}(s)\mbox{d}s
\leq \widetilde{C}t
+\widetilde{C}\eta^{-1}t
+\widetilde{C},\label{con4}
\ee
setting $ \widetilde{T}\leq\eta<1$, then we have
\be
\int_{0}^{t}I_{\eta}^{n}(s)\mbox{d}s
\leq C\widetilde{C}\label{con5}
\ee
for $t\leq \widetilde{T}$.

By (\ref{con3})-(\ref{con5}), integrating (\ref{con2}) from $[0,t]$, one derives
\be
&&\nonumber\varphi^{n+1}(t)
+\int_{0}^{t}(\|\overline{u}^{n+1}\|_{H^{1}}^{2}
+\|\overline{h}^{n+1}\|_{H^{1}}^{2}
+\|\overline{k}^{n+1}\|_{H^{1}}^{2}
+\|\overline{\varepsilon}^{n+1}\|_{H^{1}}^{2})\mbox{d}s\\
&&\nonumber\leq C\widetilde{C}\eta\bigg(\int_{0}^{t}(\|\overline{u}^{n}\|_{H^{1}}^{2}
+\|\overline{k}^{n}\|_{H^{1}}^{2}
+\|\overline{\varepsilon}^{n}\|_{H^{1}}^{2})\mbox{d}s\bigg)
\bigg[\bigg(\int_{0}^{t}I_{\eta}^{n}(s)\mbox{d}s\bigg) \exp\bigg(\int_{0}^{t}I_{\eta}^{n}(s)\mbox{d}s\bigg)+1\bigg]\\
&&\leq C\eta \exp(\widetilde{C})\int_{0}^{t}(\|\overline{u}^{n}\|_{H^{1}}^{2}
+\|\overline{k}^{n}\|_{H^{1}}^{2}
+\|\overline{\varepsilon}^{n}\|_{H^{1}}^{2})\mbox{d}s
\ee
for $T^{*}:=min\{T, \widetilde{T}\}$.

Therefore, we have
\begin{eqnarray}
&&\sum _{n=1}^{\infty}sup_{0\leq t\leq T}\varphi^{n+1}(t)
+\sum _{n=1}^{\infty}\int_{0}^{t}(\|\overline{u}^{n+1}\|_{H^{1}}^{2}
+\|\overline{h}^{n+1}\|_{H^{1}}^{2}
+\|\overline{k}^{n+1}\|_{H^{1}}^{2}
+\|\overline{\varepsilon}^{n+1}\|_{H^{1}}^{2})\mbox{d}s\nonumber\\
&&\leq C\eta \exp(\widetilde{C})\sum _{n=1}^{\infty}\int_{0}^{t}(\|\overline{u}^{n}\|_{H^{1}}^{2}
+\|\overline{k}^{n}\|_{H^{1}}^{2}
+\|\overline{\varepsilon}^{n}\|_{H^{1}}^{2})\mbox{d}s.
\end{eqnarray}
Thus, choosing $\eta$ such that $C\eta \mbox{exp}(\widetilde{C})\leq\frac{1}{2}$, one deduce
\begin{eqnarray}
&&\sum _{n=1}^{\infty}sup_{0\leq t\leq T}\varphi^{n+1}(t)
+\sum _{n=1}^{\infty}\int_{0}^{t}\|\overline{h}^{n+1}\|_{H^{1}}^{2}\mbox{d}s\nonumber\\
&&+\frac{1}{2}\sum _{n=1}^{\infty}\int_{0}^{t}(\|\overline{u}^{n+1}\|_{H^{1}}^{2}
+\|\overline{k}^{n+1}\|_{H^{1}}^{2}
+\|\overline{\varepsilon}^{n+1}\|_{H^{1}}^{2})\mbox{d}s\nonumber\\
&&\leq C\widetilde{C}
<\infty.
\end{eqnarray}
Therefore, we conclude that the full sequence $(\rho^{n},u^{n},h^{n},k^{n},\varepsilon^{n})$
converges to a limit$(\rho,u,h,k,\varepsilon)$ in the following strong sense:
$\rho^{n}\rightarrow\rho$ in $L^{\infty}(0,T;L^{2}(\Omega))$;
$(u^{n},h^{n},k^{n},\varepsilon^{n})\rightarrow(u,h,k,\varepsilon)$  in $L^{2}(0,T;H^{1}(\Omega))$.
 It is easy to prove that the limit $(\rho,u,h,k,\varepsilon)$ is a weak solution to the original nonlinear
  problem. Furthermore, it follows from (\ref{another conclusion}) that $(\rho,u,h,k,\varepsilon)$  satisfies the following
  regularity estimates:
\begin{eqnarray}
&& \sup_{0\leq t\leq T^{*}}(\|\rho\|_{H^{3}}+\|\rho_{t}\|_{H^{1}})
 +\sup_{0\leq t\leq T^{*}}(\|u\|_{H^{3}}+\|k\|_{H^{2}}+\|\varepsilon\|_{H^{2}}
 +\|h\|_{H^{2}})\nonumber\\
 &&+\sup_{0\leq t\leq T^{*}}(\|\sqrt{\rho}u_{t}\|_{L^{2}}+
 \|\sqrt{\rho}h_{t}\|_{L^{2}}+\|\sqrt{\rho}k_{t}\|_{L^{2}}+\|\sqrt{\rho}\varepsilon_{t}\|_{L^{2}})\nonumber\\
 &&+\int_{0}^{T^{*}}(\|u_{t}\|_{H^{1}}^{2}
 +\|h_{t}\|_{H^{1}}^{2}
 +\|k_{t}\|_{H^{1}}^{2}
 +\|\varepsilon_{t}\|_{H^{1}}^{2}
 +\|u\|_{H^{4}}^{2}
 +\|k\|_{H^{3}}^{2})\leq \widetilde{C}<\infty.\nonumber
 \end{eqnarray}
This proves the existence of strong solution. Then, we can easily prove the
time continuity of the solution $(\rho, u, h, k, \varepsilon)$ by adapting the
arguments in \cite{C-C-K, C-K3}. Finally, we  prove the uniqueness.
In fact, assume that $(\rho_{1}, u_{1}, h_{1}, k_{1}, \varepsilon_{1})$ and
$(\rho_{2}, u_{2}, h_{2}, k_{2}, \varepsilon_{2})$ are two strong solutions
to the problem (\ref{1.1})-(\ref{1.9}) with the regularity (\ref{regularity}).
Let $(\overline{\rho}, \overline{u}, \overline{h}, \overline{k}, \overline{\varepsilon})
=(\rho_{1}-\rho_{2}, u_{1}-u_{2}, h_{1}-h_{2}, k_{1}-k_{2}, \varepsilon_{1}-\varepsilon_{2})$.
Then using the same argument as in the derivations of (\ref{drou}),
(\ref{difu-2}), (\ref{difh-2}), (\ref{difk-2}) and (\ref{dife-2}),
we can prove that
\be
\frac{\mbox{d}}{\mbox{d}t}(\|\overline{\rho}\|_{L^{2}}^{2}+
\|\sqrt{\rho_{1}}\overline{u}\|_{L^{2}}^{2}
+\|\sqrt{\rho_{1}}\overline{h}\|_{L^{2}}^{2}
+\|\sqrt{\rho_{1}}\overline{k}\|_{L^{2}}^{2}+\|\sqrt{\rho_{1}}\overline{\varepsilon}\|_{L^{2}}^{2})\nonumber\\
\leq R(t)(\|\overline{\rho}\|_{L^{2}}^{2}+
\|\sqrt{\rho_{1}}\overline{u}\|_{L^{2}}^{2}
+\|\sqrt{\rho_{1}}\overline{h}\|_{L^{2}}^{2}
+\|\sqrt{\rho_{1}}\overline{k}\|_{L^{2}}^{2}+\|\sqrt{\rho_{1}}\overline{\varepsilon}\|_{L^{2}}^{2})\nonumber
\ee
for some $R(t)\in L^{1}(0, T^{*})$.
Thus, by Gronwall's inequality, we conclude that $(\overline{\rho}, \overline{u},
 \overline{h}, \overline{k}, \overline{\varepsilon})
 =(0, 0, 0, 0, 0)$ in $(0,T^{*})\times\Omega$. This completes the proof of Theorem 3.1.
\end{proof}

 \textbf{Competing interests}

The authors declare that they have no competing interests.

 \textbf{Authors’ contributions}
 
The authors contributed equally in this article. They read and approved the final manuscript.

 \textbf{Acknowledgements} 
 
 The research of B Yuan
was partially supported by the National Natural Science Foundation
of China (Grant No. 11471103).


\end{document}